\documentclass{article}
\usepackage[pdftitle={Paper}, pdfauthor={Alejandro Cabrera}]{hyperref}
\usepackage{tabularx}
\usepackage{amsmath}
\usepackage{amsfonts}
\usepackage{amssymb,latexsym}
\usepackage{txfonts}
\usepackage[all]{xy}
\usepackage{graphicx}			
\usepackage{xcolor}
\usepackage{tikz-cd}

\newtheorem{lemma} {Lemma} [section]
\newtheorem{proposition} [lemma] {Proposition}

\newtheorem{theorem} [lemma] {Theorem}

\newtheorem{corollary} [lemma] {Corollary}
\newtheorem{definition}[lemma] {Definition}

\newtheorem{example}[lemma] {Example}
\newtheorem{remark}[lemma]{Remark}
\newtheorem{remarks}[lemma]{Remarks}
\newenvironment{proof}{{\sc Proof:}}{ {\hspace*{\fill} $\square$\\} }

\newcommand{\dto}{\dashrightarrow}

\def\R{\mathbb{R}}

\def\h{\hbar}
\def\e{\epsilon}

\def\G{\mathcal{G}}

\def\e{\epsilon}

\def\g{\gamma}

\def\tx{\tilde{x}}

\def\LL{\mathbb{L}}

\def\C{\mathbb{C}}
\def\dE{\delta_l}
\def\uC{\underline{\C}}
\def\gr{\mathrm{gr}}

\def\rra{\rightrightarrows}
\newcommand{\comp}[1]{{(#1)}}

\newcommand{\ol}[1]{\overline{#1}}

\def\R{\mathbb{R}}
\def\g{\mathfrak{g}}
\def\h{\hbar}
\def\dto{\dashrightarrow}

\def\compo{^{(2)}}
\def\rra{\rightrightarrows}
\def\H{\mathcal{H}}

\def\G{\mathcal{G}}
\def\fih{\frac{i}{\h}}
\def\e{\epsilon}
\def\tx{\tilde{x}}
\def\tmu{\tilde{\mu}}
\def\gS{\gamma_S}
\def\factor{\kappa}
\long\def\comentar#1{}
\def\F{\mathfrak{F}_\h}

\def\id{\mathrm{Id}}

\def\hl{h^L}
\def\hr{h^R}
\def\Sl{\Sigma^L}
\def\Sr{\Sigma^{R}}
\def\phih{\phi^h}

\def\K1l{K_{\text{1-loop}}}

\newcommand{\titpar}[1]{\subsubsection{#1}}
\newcommand{\stitpar}[1]{\medskip {\bf #1}}

\begin{document}

\title{Associative half-densities on symplectic groupoids and quantization\thanks{Work primarily conducted at Instituto de Matemática, Universidade Federal do Rio de Janeiro - UFRJ, Brasil.}}

\author{Alejandro Cabrera\footnote{Department of Mathematics (MAT), Universitat Polit\`ecnica de Catalunya - BarcelonaTech (UPC); \\ ORCID ID: 0000-0003-3279-0062; \emph{emails:} alejandro@matematica.ufrj.br, alejandro.cabrera1@upc.edu }
\and Gabriel Ledesma\footnote{Instituto de Computação, Universidade Federal Fluminense - UFF, Brasil;\\ ORCID ID: 0009-0009-4687-7952; \emph{email:}gabrielgonzalo@id.uff.br}}
\date{%
    \vskip.5cm \today
}

\maketitle

\begin{abstract}
In this paper, we study half-densities enhancing the multiplication map on a symplectic groupoid and which satisfy a suitable associativity condition. This is structurally motivated by the expected complete semiclassical-analytic approximation to a star product for the underlying Poisson manifold. We show the existence and classification of such associative half-densities, and further apply this theory to the understanding of semiclassical factors in Kontsevich's quantization formula. In the particular case of a linear Poisson structure, we recover the factors appearing in the Duflo isomorphism and its Kashiwara-Vergne extensions as a canonical associative enhancement.
\end{abstract}

\setcounter{tocdepth}{2}
\tableofcontents

\section{Introduction}
This paper can be framed in the context of the study of analytic, non-formal and geometric aspects of quantization of Poisson structures by star products. We study \emph{associative half-densities} enhancing symplectic groupoid multiplication and their role in the quantization of the underlying Poisson manifold (see \cite{BFFLS,Kont} and \cite{CFer1} for non-formal aspects).
We begin this introduction explaining the general context (see \cite{BatWei, GSbook, KarMas}) and, then, move towards the specific topics and contributions. 


\stitpar{Quantization of Poisson manifolds.} The problem of quantization of Poisson brackets is a very important and well-known one, having multiple formulations and branches.
Consider $(M,\pi\equiv \{,\})$ a Poisson manifold. The quantization that we shall focus on is motivated by symbol calculus and based on a family of operations $\h\mapsto \star_\h$, called \emph{star product} among \emph{classical symbols} in $C^\infty(M)$. It can be axiomatized as follows: (see \cite{BFFLS,KarMas} for variants of these axioms) 
\begin{itemize}
\item[S1)] $f_1 \star_\h f_2 = f_1f_2 + O(\h)$
\item[S2)] $\frac{i}{\h}(f_1 \star_\h f_2-f_2 \star_\h f_1) = \{f_1,f_2\} + O(\h)$
\item[S3)] $(f_1 \star_\h f_2)\star_\h f_3 = f_1 \star_\h (f_2 \star_\h f_3) + O(\h^\infty)$.
\end{itemize}
Here $\h$ is Planck's constant seen as a scale parameter and $\h\to 0$ is the \emph{classical limit} in which the quantum scales tend to zero.
%
%
Given a Poisson manifold $(M,\pi)$, its \emph{quantization problem} is to study the existence, classification and representations of such a star product $\star_\h$ with the above properties. For $M=T^*X$ with canonical brackets, this problem is solved by the $\star_\h$ coming from symbol calculus for pseudo-differential operators on $X$ (see \cite{GSbook,Zw}). For arbitrary Poisson manifolds, the mere existence problem is notably hard.

When working purely with asymptotic expansions as $\h\to 0$, i.e. modulo $O(\h^\infty)$, the above three axioms determine a so-called \emph{formal star product} (formal deformation quantization, \cite{BFFLS}) and their existence and classification was understood through the well-known and deep work of Kontsevich \cite{Kont}. The most basic element in Kontsevich's work is an explicit formula for a star product $\star_\h^K$ in the case of $M=\R^n$ a coordinate domain. For later reference, the structure of this formula can be described in a factored form (see \cite[\S 7]{CatDheFel}):
\begin{equation}\label{eq:kontstarfactors}
(e^{\fih \xi_1}\star^K_\h e^{\fih \xi_2})(x) = (a^K_0(\xi_1,\xi_2,x) + \h a^K_1(\xi_1,\xi_2,x) + \h^2 a^K_2(\xi_1,\xi_2,x) + \dots) e^{\fih S_K(\xi_1,\xi_2,x)} 
  \end{equation}
where $\xi_1,\xi_2:M\to \R$ are linear functions, $x\in M$, and $S_K, a^K_j=e^{K_{(j+1)-loop}} \in C^\infty(M)[[\xi_1,\xi_2]]$ are formal expansions in the Fourier dual variables $\xi_1,\xi_2$. Here, $S_K, K_{\text{$l$-loop}}$  are given by formal sums over Kontsevich graphs with $0$-loops and $l$-loops, respectively. The theory of this paper aims towards understanding $a_0^K=e^{\K1l}$ and its relation to $S_K$.


\stitpar{Symplectic groupoids as part of a semiclassical approximation.} We now outline the role of symplectic groupoids in the quantization problem (\cite{Karasev,We91,Zak}) following accounts in \cite{BatWei,We87} and references therein, which can be seen as providing a general theory for the factor $S_K$ above. 
To that end, let us recall a refined quantum-classical correspondence for the symplectic case in which state spaces are included as in the following table.
\begin{center}
\begin{tabular}{ |p{4cm}|p{4cm}|p{4cm}| }
\hline
\multicolumn{3}{|c|}{semiclassical approximation correspondence} \\
 \hline
\emph{Quantum} & \emph{Enhanced Symplectic} & \emph{Symplectic}  \\ 
 \hline
$\mathcal{H}$: $\mathbb{C}$-vector (Hilbert) space & $(S,\omega)$: symplectic &  $(S,\omega)$: symplectic  \\ 
 \hline
 $\psi_\h \in \mathcal{H}$ element or "WKB-state" & $L\hookrightarrow (S,\omega)$ Lagrangian submanifold \emph{with $\sigma$ a half-density along $L$} & $L\hookrightarrow (S,\omega)$ Lagrangian submanifold \\  
 \hline
\end{tabular}
\end{center}
We shall first use the first and third columns. Going back to quantization of Poisson manifolds, applying the above correspondence to a general family of algebra structures on some $\H$,
\[ \star_\h : \H \otimes \H \to \H, \ \iff \psi_\h \in \H^* \otimes \H^* \otimes \H \overset{\text{SC lim}}{\to} [L \hookrightarrow (S,-\omega) \times (S,-\omega) \times (S,\omega)],\]
one proposes that such $(S,\omega)$ must be the space of morphisms of a \emph{symplectic groupoid} $(G\rightrightarrows M, \omega)$ where $L=\gr(m)$ is the graph of its partially defined multiplication operation. This operation must be associative as a semiclassical counterpart to axiom (S3). The appearance of $M$ as objects of $G$ can be motivated similarly by considering that $\star_\h$ has approximate units (from axiom (S1)). Alternative reasonings leading to the same notion can be found in \cite{KarMas}.

Such insights motivated the study of symplectic groupoids, leading to a deep understanding of their structure and to uncover an underlying Lie theory for Poisson brackets (see \cite{CrFeMa} and references therein). For any symplectic groupoid  $(G\rightrightarrows M, \omega)$, the manifold of objects $M$ inherits a unique Poisson structure $\{,\}$ and one says that $(G,\omega)$ \emph{integrates} $(M,\{,\})$, just as Lie groups integrate Lie algebras. Moreover, after Kontsevich's construction of a formal star product $\star_\h^K$ for any Poisson manifold in \cite{Kont}, it was understood that the \emph{$0$-loop factor} $S_K$ in \eqref{eq:kontstarfactors} is indeed completely determined by an underlying (local) symplectic groupoid structure, see \cite{CDh,C22} (also \cite{CatDheFel,Karabegov} for direct formal-family treatments). Alternative constructions relating symplectic groupoids and noncommutative algebras can be also found in \cite{Hawk,Rieff2,GuWi} (see also below). 


\stitpar{The full data of the semiclassical approximation: semiclassical analysis and half-densities.}
We now go back to the semiclassical approximation correspondence and focus on the second \emph{enhanced symplectic} column. This column is motivated by certain important facts in \emph{semiclassical analysis} (see \cite{GSbook,Mein,Zw} and a similar table in \cite[Sec. 2]{BatWei}) and the key new ingredient is that the Lagrangian submanifolds come enhanced with a half-density defined on them. These half-densities are very important factors in the cases where $\H=L^2(X)$ and in which the theory is driven towards understanding asymptotic behaviour of solutions of PDEs such as the Schrodinger equation with $\h$ as a scale parameter (see \cite{Zw} when $X=\R^n$). 

When applied to star products, this correspondence inspires a concrete \emph{enhancement} of the notion of symplectic groupoid which is the main object of study in this paper. The data of an \emph{enhanced symplectic groupoid} $(G\rightrightarrows M,\omega,\sigma)$ is then a symplectic groupoid together with a half-density $\sigma$ defined along the graph $\gr(m)$ of its multiplication map. This   structure is designed to encode the complete semiclassical data behind a star product $\star_\h$ and thus goes deeper into the structural understanding of quantization of Poisson manifolds. In terms of Kontsevich's $\star_\h^K$, the novel structure $\sigma$ points towards the understanding of the \emph{$1$-loop factor} $a_0^K$ in \eqref{eq:kontstarfactors}.

Half-densities also appear in the construction \emph{convolution $C^\ast$-algebras} $C(H)$, now suitably defined along the total space of a Lie groupoid $H$ (see \cite{Lesc,Sta1,Sta2}). These are related to star products for $M=A_H^*$, with $A_H$ the Lie algebroid of $H$, once a quantization map $Q_\h:C^\infty(A_H^*) \to C(H)$ is constructed so that $Q_\h(f_1)\circ Q_\h(f_2) = Q_\h(f_1\star_\h f_2)$. This map typically involves Fourier transform and auxiliary data, leading to an enhancement of the cotangent groupoid $G=T^*H$ (see also \cite[Appendix]{Sta1}). For example, the construction of such $Q_\h$ is detailed in \cite{Rieff2} for $H$ a Lie group, and in \cite{KarOsb} for $A_H^*$ a magnetic cotangent bundle. The first case can be seen to yield a family of star products on $M=\mathfrak{h}^*$ which will be studied in detail in Section \ref{subsec:linear} below. (See also Example \ref{ex:symplectic} below for a relation to the second case and \cite{CFer2} for general $H$.)


\stitpar{Associativity and the enhanced symplectic category.}
To complete the definition of an enhanced symplectic groupoid $(G\rightrightarrows M,\omega,\sigma)$, we finally discuss the associativity condition on $\sigma$ which corresponds to the semiclassical approximation of the associativity axiom (S3) for a star product. This is expressed using a composition operation for enhanced canonical relations $(L,\sigma):(S_1,\omega_1) \dto (S_2,\omega_2)$ which corresponds to the composition of (semiclassical) Fourier Integral operators $F_\h$ under Stationary Phase approximations, see \cite{GSbook,Mein}.
%
%
Lagrangians compose as relations $L_1 \circ L_2$, which requires a certain transversality or cleanness hypothesis for the result to be Lagrangian again and, in that case, the corresponding half-densities can also be composed yielding $\sigma_1\circ \sigma_2$ a half-density on $L_1\circ L_2$.

Coming back to our $(G\rightrightarrows M,\omega,\sigma)$, the relevant \emph{associativity condition} then reads (see Definition \ref{def:main} below)
\[ (\gr(m),\sigma) \circ ((\gr(m),\sigma) \times \id) = (\gr(m),\sigma) \circ ( \id \times (\gr(m),\sigma)). \]
We observe that this is a non-linear equation for $\sigma$ which admits non-trivial solutions even for simple groupoids (see Example \ref{ex:pi=0}). Solutions to this equation will be called \emph{associative half-densities} (or \emph{associative enhancements}) of the underlying symplectic groupoid and are our main objects of study.


\stitpar{Main results.} 
In this paper, we thus develop the theory of such associative half-densities $\sigma$ enhancing symplectic groupoids and apply it to study the semiclassical factors $S_K$ and $a_0^K$ behind Kontsevich's star product formula \eqref{eq:kontstarfactors}. Concretely, besides the general motivations given above, our main results are the following. 

\medskip
\noindent\emph{First main result:} (Existence and classification, Theorem \ref{thm:main}) Every symplectic groupoid $(G\rra M, \omega)$ admits an associative enhancement $\sigma$. Moreover, the choice of a non-vanishing half-density $\mu\neq 0$ on $M$ induces both a canonical associative enhancement $\sigma^c$ of $G$ and an identification of the set of non-vanishing associative enhancements $\sigma\neq 0$, modulo a natural notion of equivalence (Definition \ref{def:equivenhanc}), with the second multiplicative cohomology group $H^2(G,\C^*)$ of $G$. 
\medskip

We also show that a $\mu$ as above can be recovered from any $\sigma\neq 0$, that it can be seen as an enhancement of the identity $1:M\hookrightarrow G$, and that it satisfies an enhanced identity axiom (Proposition \ref{prop:idaxiom}).
Next, moving towards the study of $\star^K_\h$ in \eqref{eq:kontstarfactors}, we recall from \cite{C22,CatDheFel} that a coordinate Poisson manifold can be set into a family $(M=\R^n,\epsilon \pi)$, with formal parameter $\e$, and integrated by a formal family of symplectic groupoids $(G_K\rra M,\omega_c)$ defined by $S_K$. On the one hand, the first result above allows us to construct a formal family of canonical associative enhancements $\sigma^c$ for this $G_K$ out of the coordinate half-density $\mu=|dx|^{1/2}$ on $M$. We see this $\sigma^c$ as being defined by the underlying coordinate Poisson geometry through our general theory. On the other hand, the 1-loop factor $a_0^K$ in Kontsevich's quantization formula defines another formal family of enhancements $\sigma^K$ of $G_K$ (see eq. \eqref{eq:Jsigma}).

\medskip
\noindent\emph{Second main result:} (characterization of Kontsevich 1-loop factor, Theorem \ref{thm:Kenhancement}) $\sigma^K$ and $\sigma^c$ are equivalent formal families of enhancements. 
\medskip

Altogether, these provide a structural explanation of the semiclassical associativity properties of the 1-loop factor $a_0^K$ in Kontsevich formula. We also prove a stronger version $\sigma^K=\sigma^c$ of the second result in the case of linear Poisson manifolds $M=\g^*$ (Proposition \ref{prop:gastsigmaK}). In this case, we recall that this factor $a^K_0$ plays an important role in leading to fine properties of the center of $\star^K$ and is related to the Duflo isomorphism (see \cite[\S 8.3]{Kont} and \cite{Amar}) and its extension to certain convolution algebras, as posed by Kashiwara-Vergne (\cite{KV,AnSaTo}). Then, our results specialized to the case $M=\g^*$ provide an interpretation for the underlying key square-root-Jacobian factors appearing in $a_0^K$ in terms of our general theory of associative half-densities.


\stitpar{Contents.} We detail the contents further as follows.
\begin{itemize}
    \item In Section \ref{sec:half}, after recalling preliminary concepts in \S \ref{subsec:def}, we provide the main definition (Definition \ref{def:main}) of enhanced symplectic groupoid. In \S \ref{subsec:exist}, we prove the main theorem (Theorem \ref{thm:main}) providing existence and classification for such structures.  In \S \ref{subsec:props}, we also provide several additional properties (including the enhanced identity axiom in Proposition \ref{prop:idaxiom}) and illustrative examples. 

    \item In Section \ref{sec:app}, we return to Kontsevich's quantization formula \eqref{eq:kontstarfactors} in $M=\R^n$. In \S \ref{subsec:coordstar}, we first follow \cite{C22} to describe an underlying local symplectic groupoid $G_\pi$ and characterize its possible enhancements including the canonical one $\sigma^c$. 
In \S \ref{subsec:Khalf}, we then describe the formal families appearing as the asymptotic expansion of $G_{\e\pi}$ as $\e\to 0$, which recovers the formal family $G_K$ mentioned above. We also discuss formal families of enhancements including the $\sigma^K$ coming from $a_0^K$. We then prove the second  main result (Theorem \ref{thm:Kenhancement}) stating that $\sigma^K$ and $\sigma^c$ are equivalent formal families of enhancements.
Finally, in \S \ref{subsec:linear} we illustrate these results in the case $M=\g^*$ and prove Proposition \ref{prop:gastsigmaK} stating that $\sigma^K=\sigma^c$ in this case.
\end{itemize}


\stitpar{Line bundle valued half-densities.} In the main text, we focus on scalar valued half-densities. This is enough for our purposes since our main objectives are, in the terminology of semiclassical analysis (see \cite{GSbook}), of microlocal nature and we can thus restrict to neighborhoods of identities in the underlying symplectic groupoids. Nevertheless, we discuss in Appendix \ref{app:E} the extension to the line bundle valued case generalizing what happens with  global symbols of FIOs (\cite{GSbook,Mein}). This setting can be relevant for other types of quantization, for example, semiclassical versions of geometric quantization of symplectic $M$ with complex polarization. Such cases will be explored elsewhere.

\medskip

{\bf Acknowledgements.} The authors want to thank R. Fernandes and E. Meinrenken for useful conversations and suggestions, to N. Moshayedi for useful comments on the first version, and to J. Grabowski for pointing to relevant literature.
A.C.'s research  was partially supported by the grants CNPq PQ 309847/2021-4, CNPq Universal 402320/2023-9 and FAPERJ CNE E-26/204.097/2024. G.L. thanks CAPES for support during his PhD at Universidade Federal do Rio de Janeiro where much of this work was developed.

\section{Associative half-densities over groupoid multiplication}\label{sec:half}
In \S \ref{subsec:def}, we introduce our main object of study given by associative half-densities. We first recall half-densities and their composition along canonical relations following \cite{GSbook} (see also \cite{Mein}), while we assume the reader is familiar with basic notions of symplectic groupoids (see \cite{CrFeMa,Macbook}). In \S \ref{subsec:exist} we show the main result (Theorem \ref{thm:main}) providing existence and classification for these structures. Finally, in \S\ref{subsec:props} we deduce some general properties and illustrate with simple examples.

\subsection{Definition and the associativity condition}
\label{subsec:def}
Here, we move towards the definition of an associative half-density enhancing a symplectic groupoid multiplication (Definition \ref{def:main}).

\medskip

\titpar{Properties of $\alpha$-densities.} Let us first recall some basic facts about half-densities. For a vector space $V$ of dimension $n$, we denote $\beta:\R^n \overset{\sim}{\to} V$ a choice of linear basis. Then, the vector space $|V|^\alpha$ of {\bf $\alpha$-densities} consists of functions $\beta\mapsto \sigma(\beta)\in \C$ defined on linear bases $\beta$ and such that $\sigma(\beta \circ A)=|det(A)|^\alpha \sigma(\beta), \ A\in Gl(n)$. The set of {\bf half-densities} on $V$ is denoted $|V|^{1/2}$. Notice that, since the $Gl(n)$-action on bases is transitive, two $\alpha$-densities coincide $\sigma_1=\sigma_2$ iff they take the same value on some particular basis $\beta$.
\begin{example}
Given a symplectic vector space $(V,\omega)$, the \emph{Liouville half-density} $\lambda_V \in |V|^{1/2}$ is defined as
\begin{equation}\label{eq:liouvillehalf} 
\lambda_V =\left|\frac{\omega^{n}}{n!}\right|^{1/2}, \ dim(V)=2n.
\end{equation}
\end{example}
An isomorphism $\phi:V_1 \to V_2$ induces a bijection $\phi^*:|V_2|^\alpha \to |V_1|^\alpha$ via $(\phi^*\sigma_2)(\beta_1)=\sigma_2(\phi\circ \beta_1)$. A property that we shall often use is the following: for an exact sequence
\[ 0 \to V_1 \to V \to V_2 \to 0 \]
we have a natural isomorphism
\begin{equation}\label{eq:sesV}
|V_1|^\alpha \otimes |V_2|^\alpha \simeq |V|^\alpha, 
\end{equation}
where we denote resulting elements $\sigma=\sigma_1 \otimes \sigma_2 \in |V|^\alpha$ defined by $(\sigma_1 \otimes \sigma_2)(\beta_1 \oplus \beta_2) = \sigma_1(\beta_1)\sigma_2(\beta_2)$ for $\beta_1$ a basis of $V_1$ seen injected in $V$ and $\beta_2$ any complementing l.i. subset in $V$ whose projection onto $V_2$ is a basis. We recall that the fact that this is well defined comes from the factorization property for the determinant of block-triangular matrices. When $\sigma_1\neq 0$ or $\sigma_2 \neq 0$, we write
\[ \sigma_2 = \sigma/\sigma_1 \text{ or } \ \sigma_1 = \sigma/\sigma_2, \text{ respectively.} \] 

\medskip

\titpar{Composition in the linear symplectic category.} 
Let us consider symplectic vector spaces $V\equiv (V,\omega)$ and denote $\overline{V}\equiv (V,-\omega)$. Linear canonical relations $L:(V_1,\omega_1)\dto (V_2,\omega_2)$ between them are given by Lagrangian subspaces $L\subset \overline{V}_1 \times V_2$ and they can be composed as ordinary relations between sets yielding $L_2\circ L_1: V_1 \dto V_3$ for given canonical relations
$$ V_1 \overset{L_1}{\dto} V_2 \overset{L_2}{\dto} V_3 .$$
The fact that  $L_2\circ L_1$ is Lagrangian in $\ol{V}_1 \times V_3$ can be deduced from the following description, which we also independently need.
Let us consider $L_2 \ast L_1 = \{ (v_1,v_2,v_2,v_3): (v_1,v_2)\in L_1, (v_2,v_3)\in L_2 \}$ inside $V_1\times V_2\times V_2 \times V_3$ and the map
\begin{equation}\label{eq:alpha}
\alpha: L_2 \ast L_1 \to L_2\circ L_1, \ (v_1,v_2,v_2,v_3)\mapsto (v_1,v_3).
\end{equation}
Notice that $L_2\ast L_1 = C \cap (L_1 \times L_2)$ where $C=V_1 \times \Delta_{V_2} \times V_3$, with $\Delta_x=(x,x)$ the diagonal, is coisotropic inside $\ol{V}_1\times V_2 \times \ol{V}_2 \times V_3$. Hence $L_2\circ L_1$ is the corresponding symplectic reduction of $L_1\times L_2$, explaining why it is Lagrangian. 

Given half-densities $\sigma_j\in |L_j|^{1/2}, \ j=1,2$, we think of $(L_j,\sigma_j): V_j \dto V_{j+1}$ as \emph{enhanced morphisms} and define their composition as
\[ (L_2,\sigma_2) \circ (L_1,\sigma_1) = (L_2\circ L_1, \sigma_2 \circ \sigma_1) \]
where $\sigma_2 \circ \sigma_1$ is an underlying \emph{composition operation} (\cite[\S 7.1-7.2]{GSbook} and \cite[\S 3]{Mein}) for half-densities over linear canonical relations. We shall only need the explicit description of this operation in the particular case in which $Ker(\alpha)=0$, which works as follows.
Consider the short exact sequence
\[ 0 \to L_2\ast L_1 \to L_1 \times L_2 \overset{\tau}{\to} V_2 \to 0\]
with $\tau(v_1,v_2,v'_2,v_3)= v_2 - v_2'$. Since $Ker(\alpha)=0$, we have $\alpha: L_2\ast L_1 \simeq L_2\circ L_1$ is a natural isomorphism, which we omit from the notation, and define
\begin{equation}\label{eq:circsigmas}
\sigma_2 \circ \sigma_1 := (\sigma_1 \times \sigma_2)/\lambda_{V_2}
\end{equation}
for the Liouville half-density $\lambda_{V_2}$ on $V_2$. 
\begin{remark}\label{rmk:linearcompoformula}
More concretely, suppose we have a basis $\beta_\ast$ for $L_2\ast L_1$ and $\beta_2$ an complementing l.i. set in $L_1\times L_2$ such that $\tau\circ \beta_2$ is a basis of $V_2$. Assume further that $(\tilde\beta_1 \times \tilde \beta_2)\circ A=\beta_\ast \cup \beta_2$ for some bases $\tilde \beta_j$ of $L_j, \ j=1,2$, and some change of basis matrix $A$. Then,
\[ (\sigma_2\circ \sigma_1)(\alpha (\beta_\ast)) = \frac{(\sigma_1\times \sigma_2)(\beta_\ast \cup \beta_2)}{\lambda_{V_2}(\tau(\beta_2))} = |det(A)|^{1/2} \frac{\sigma_1(\tilde\beta_1)\sigma_2(\tilde\beta_2)}{\lambda_{V_2}(\tau(\beta_2))} \]
\end{remark}

\medskip

\titpar{Manifolds, transversality and composition of enhanced canonical relations.}
Here, we consider symplectic manifolds $S\equiv (S,\omega)$ and canonical relations $L:S_1\dto S_2$ between them given by Lagrangian submanifolds $L\hookrightarrow \ol{S}_1 \times S_2$. In this context, the set-theoretic compostion of relations $L_2\circ L_1$ may fail to define a Lagrangian submanifold, so that they form a "partial" category. A general condition which ensures that the composition is again a canonical relation is \emph{clean composition}, but for us it will suffice to recall the stronger \emph{transverse composition} condition, following \cite[\S 4]{GSbook} (see also \cite{Mein}). Given $L_j:S_j\dto S_{j+1},\ j=1,2$ canonical relations, they are said to have transverse composition when the intersection of submanifolds 
\[ L_2\ast L_1:= (L_1\times L_2)\cap (S_1\times \Delta_{S_2}\times S_3)\]
is transverse inside $S_1\times S_2\times S_2 \times S_3$ and the map $\alpha: L_2\ast L_1 \to L_2\circ L_1, (v_1,v_2,v_2,v_3)\mapsto (v_1,v_3)$ is proper with connected fibers.
The transversality condition implies that, at the linear level of tangent spaces, $Ker(D\alpha)=0$ so that we can apply the formulas recalled above.

\emph{Enhanced canonical relations} $(L,\sigma):S_1\to S_{2}$ are given by pairs consisting of a canonical relation $L$ and a half-density $\sigma \in \Gamma |TL|^{1/2}$ on it. When a composition $L_2\circ L_1$ is transverse (or, more generally, clean), we can define the composition $\sigma_2\circ \sigma_1$ as in the linear case thinking of the corresponding tangent spaces. In this way, we obtain a (partial) \emph{composition law in the enhanced symplectic category},
\[ (L_2,\sigma_2)\circ (L_1,\sigma_1) = (L_2\circ L_1,\sigma_2\circ \sigma_1). \]
As mentioned in the introduction, this composition law corresponds to that of Fourier Integral Operators under semiclassical limit (\cite[\S 8]{GSbook} and \cite{Mein}). The identity for this operation is $\id_S = (\gr(id_S),\lambda_S)$ where $id_S:S\to S$ is the identity map and $\lambda_S$ is the Liouville half-density \eqref{eq:liouvillehalf} on each tangent space. 

We shall be especially interested on graphs $\gr(f):S_1\dto S_2$ of maps $f:D\subset S_1 \to S_2$ defined on a submanifold $D\subset S_1$. These always compose transversely and we note further that when $\gr(f)$ is a canonical relation, then $f:D\to S_2$ must be a submersion. In these cases, since $\gr(f)\simeq D$ as manifolds, we shall identify enhancements of $\gr(f)$ with $\sigma \in \Gamma |TD|^{1/2}$. For $f_j:D_j\subset S_j \to S_{j+1}, \ j=1,2$, the composition law for enhancements $\sigma_j \in |TD_j|^{1/2}$ is given as follows. Using the notation $[V]$ to indicate a basis of a vector space $V$,
\begin{equation}\label{eq:compofs}
(\sigma_2\circ \sigma_1)([T_xD_0]) = \frac{\sigma_2([T_{f_1(x)}D_2]) \sigma_1([T_x D_0]\cup [C_x])}{\lambda_{S_2}([T_{f_1(x)}D_2] \cup D_xf_1([C_x])) }
\end{equation}
for each $x \in D_0:=f_1^{-1}(D_2)\subset D_1$, for any choice of complement $T_x D_1 = T_x D_0 \oplus C_x$ and for any choice of bases $[T_{f_1(x)}D_2]$ and $[C_x]$.

\titpar{The definition of associative half-densities.}
Let $(G\rra M, \omega)$ be a given symplectic groupoid (see \cite{CDW,Macbook} for the general definitions) and recall the enhanced identity morphism $\id_G=(\gr(id_G),\lambda_G)$ defined by the Liouville half-density \eqref{eq:liouvillehalf} on $(G,\omega)$.
The following is the main definition of this paper.
\begin{definition}\label{def:main}
    An \textbf{associative half-density} (or \emph{enhancement}) on $(G\rra M,\omega)$ is the data of a half-density $\sigma \in \Gamma |T(\gr(m))|^{1/2}$ defined along the graph of the multiplication map $m$ such that the corresponding enhanced canonical relation
        \[ (\gr(m),\sigma):(G,\omega)\times (G,\omega) \dto (G,\omega) \]
satisfies the following associativity axiom: the diagram 
\begin{equation}\label{eq:assocECR}
        \begin{matrix}
     (G,\omega)\times (G,\omega) \times (G,\omega)   &\overset{(gr(m),\sigma)\times \id_G}{\dto} & (G,\omega)\times (G,\omega) \\
     & & \\
     \id_G \times (gr(m),\sigma) \overset{|}{\rotatebox{-90}{$\dashrightarrow$}} &  & \overset{|}{\rotatebox{-90}{$\dashrightarrow$}} (gr(m),\sigma) \\
     (G,\omega)\times (G,\omega) & \overset{(gr(m),\sigma)}{\dto} & (G,\omega)
    \end{matrix}
        \end{equation}  
     commutes in the enhanced symplectic category. 
        We say that $\sigma$ is \textbf{nonvanishing} when $\sigma|_z \neq 0$ for every $z\in \gr(m)$. We refer to the data $(G\rra M,\omega,\sigma)$ as to an {\bf enhanced symplectic groupoid}.
\end{definition}

Let us recall the notation $G^{(k)}\subset G^k$ for strings $(g_1,\dots,g_k)$ of $k$ composable arrows, $s(g_j)=t(g_{j+1})$.
Since $m:G^{(2)}\subset G\times G \to G$ is a map, as mentioned earlier, we think of $\sigma$ as living on the domain consisting of composable arrows, $\sigma \in \Gamma |TG^{(2)}|^{1/2}$. Finally, notice that since $m$ is associative by definition of $G\rra M$, the only non-trivial axiom in the definition is the following {\bf associativity equation} for $\sigma$,
\begin{equation}\label{eq:assocsigma}
\sigma \circ (\sigma \times \lambda_G) = \sigma \circ (\lambda_G \times \sigma) \in |TG^{(3)}|^{1/2}
\end{equation}
where $\circ$ denotes composition of half-densities and $\times$ the half density associated with a product.
\begin{remark} (Coverings)\label{rmk:coverings}
If $\phi:G' \to G$ is a morphism of Lie groupoids which induces $id_M$ on objects and defines a covering on $s$-fibers, then both $\omega$ and $\sigma$ can be naturally lifted from $G$ to $G'$ yielding an enhanced symplectic groupoid structure on $G'$.
\end{remark}

\titpar{Morphisms and equivalences.} A morphism between enhanced symplectic groupoids $$(L,\gamma): (G\rra M,\omega,\sigma) \to (G'\rra M', \omega',\sigma')$$
is an enhanced canonical relation $(L,\gamma)$ between the underlying symplectic manifolds such that the compositions $(m',\sigma')\circ (L\times L,\gamma\times \gamma)$ and $(L,\gamma)\circ (m,\sigma)$ are clean and yield the same result in the enhanced category. 
For two enhancements on the same $(G\rra M,\omega)$ we consider a more restricted class of equivalences with $L=\gr(id_G)$ and $\gamma=\kappa \lambda_G$, as follows. 
\begin{definition}\label{def:equivenhanc}
Two enhancements $\sigma$ and $\sigma'$ of a symplectic groupoid $(G\rra M,\omega)$ are (simply) \textbf{equivalent} if there exists a smooth nonvanishing function $\factor:G \to \C^*$ such that
\[ \sigma'|_{(g_1,g_2)} = \frac{\factor(g_1)\factor(g_2)}{\factor(g_1g_2)} \sigma|_{(g_1,g_2)}.\]
We use the notation $[\sigma]$ for the corresponding equivalence class of enhancements on $(G\rra M,\omega)$ defined by $\sigma$.
\end{definition}
It will be clear after next subsection that if $\sigma$ satisfies the associativity condition, all the equivalent ones also do.

\begin{remarks}\label{rmk:variosconvolbisec}
We collect here some remarks about the definitions. First, when the data $(\gr(m),\sigma)$ comes from a star product $\star_\h$, the axiom (S1) for $\star_\h$ implies $\sigma|_{1\compo}\neq 0$ and we shall show in \S\ref{subsec:props} that this implies $\sigma\neq 0$ is globally non-vanishing when $G$ is $s$-connected. Second, Maslov line bundle valued half-densities appear naturally when $\star_\h$ is a Fourier integral operator \cite{GSbook,Mein} (see also \cite{CFer1,CFer2}). This type of half-densities is discussed in Appendix \ref{app:E}. Finally, recall that multiplication $m$ induces a composition (or "convolution") operation on \emph{Lagrangian bisections} $L\hookrightarrow (G,\omega)$. The enhancement $\sigma$ of $m$ allows to extend this operation to enhanced Lagrangian bisections $(L,\rho)$,
\[ (L_1,\rho_1), (L_2,\rho_2) \mapsto (\gr(m),\sigma)\circ ( (L_1,\rho_1) \times (L_2,\rho_2) ) \]
where we see $(L_1,\rho_1)\times (L_2,\rho_1):\ast \dto G\times G$ as morphisms from a point space and
in which the composition is always transverse. This can be seen as the semiclassical approximation to the $\star_\h$-product of WKB states corresponding to the $(L_j,\rho_j)$ (see also \cite{KarMas}).
\end{remarks}

\subsection{Existence and classification}
\label{subsec:exist}

For this subsection, let us fix a symplectic groupoid $(G\rra M,\omega)$. We aim at proving the main existence and classification result for associative enhancements, Theorem \ref{thm:main} below.

\titpar{The associativity equation in split form.}
To make the associativity condition more explicit, we aim at decomposing the tangent directions in $TG\compo$ into a sum of three contributions: those keeping the source, those keeping the target, and those moving the underlying point in $M$.
 To this end, let us first recall how \emph{splittings} can decompose tangent directions $TG$ in a groupoid and, after that, how to use them to decompose $TG\compo$ via \eqref{eq:phih} below, as wanted, leading to a split version of associativity in Proposition \ref{prop:crit2}.

Given a Lie groupoid $G\rra M$, we consider the short exact sequences:
\[ 0 \to Ker(Ts) \to TG \overset{Ts}{\to} s^*TM \to 0, \ \ \  0 \to Ker(Tt) \to TG \overset{Tt}{\to} t^*TM \to 0. \]
In this context, a \textbf{splitting $h^L$ (resp. $h^R$) of $TG$ along $Ts$ (resp. $Tt$)} is defined to be a vector bundle morphism over $id_G$ which splits the above sequence,
$$ \hl: s^*TM \to TG, \text{ (resp. } \hr: t^*TM \to TG\text{)} $$ 
and such that it reduces to $T1$ at identity points $1_x\in G$.
In \cite{GM10}, it is shown that any Lie groupoid $G$ admits such a splitting and that, given $h^L$, one can induce a splitting $\hr$ along $Tt$ via the formula
$\hr_g(v) = Tinv(\hl_{g^{-1}}(v))$.
We shall always think that $\hr$ is defined by an $\hl$ in this way.

Let us denote by $A^s \to M$ the vector bundle with fibers $A^s_x = Ker(D_{1(x)}s), \ x\in M$ and, similarly, $A^t_x = Ker(D_{1(x)}t), \ x\in M$. We recall that a standard convention is to choose $A^s$ as the Lie algebroid of $G$, with a Lie bracket inherited from right-invariant vector fields (see \cite{Macbook}).
Aiming at our desired decomposition of $TG\compo$, given splittings $\hl,\hr$, we consider the bundle isomorphisms
$$ \Sl: t^*A^s \oplus s^*TM \to TG, \ \ \Sr: t^*TM \oplus s^*A^t \to TG $$
defined by, 
$ \Sl_g(k^s,v) = TR_g(k^s) + \hl_g(v)$ and $ \Sr_g(v, k^t) = \hr_g(v)+TL_g(k^t)$, for $g\in G$ and 
where $R_{g_1}(g_2)=g_2g_1$ denotes right multiplication and $L_{g_1}(g_2)=g_1g_2$ denotes left multiplication.
%
%
%
Finally, we introduce the following bundle isomorphism defined, for each composable pair $(g_1,g_2)\in G^{(2)}$, by 
\begin{eqnarray} \phih_{g_1,g_2}: A^s_{t(g_1)} \oplus T_{s(g_1)}M \oplus A^t_{s(g_2)} \to T_{(g_1,g_2)}G^{(2)}, \nonumber \\
\phih_{g_1,g_2}(k^s, v, \tilde k^t) = \left( TR_{g_1}(k^s)+\hl_{g_1}(v)  , TL_{g_2}(\tilde k^t ) + \hr_{g_2}(v) \right) \label{eq:phih}
\end{eqnarray}

With the above notation, we can provide an explicit characterization of the associativity condition.
\begin{proposition}\label{prop:crit2}
Let $\sigma$ be a half-density on $G^{(2)}$.
	The associativity equation \eqref{eq:assocsigma} holds iff for each $(g_1,g_2,g_3) \in G^{(3)}$ and denoting
	$$ \sigma^h|_{g_1,g_2} := (\phih_{g_1,g_2})^*\sigma|_{g_1,g_2} \in \left| A^s_{t(g_1)} \oplus T_{s(g_1)}M \oplus A^t_{s(g_2)}\right|^{1/2},$$
	we have 
	\begin{eqnarray}\label{eq:assocsimpl2}
	\frac{\sigma^h|_{g_1g_2,g_3} \left( [A^s_{t(g_1)}]\oplus [T_{s(g_2)}M] \oplus [A^t_{s(g_3)}] \right)  \sigma^h|_{g_1,g_2} \left( [A^s_{t(g_1)}]\oplus [T_{s(g_1)}M] \oplus [A^t_{s(g_2)}] \right)}{\lambda_G\left(\Sl_{g_1g_2}([A^s_{t(g_1)}]\oplus [T_{s(g_2)}M])\right)} =  \nonumber\\
	\frac{\sigma^h|_{g_2,g_3} \left( (\Sl_{g_2})^{-1}\Sr_{g_2}\left( [T_{s(g_1)}M] \oplus [A^t_{s(g_2)}] \right) \oplus [A^t_{s(g_3)}] \right)  
		\sigma^h|_{g_1,g_2g_3} \left( [A^s_{t(g_1)}]\oplus [T_{s(g_1)}M] \oplus [A^t_{s(g_3)}] \right)}{\lambda_G\left(\Sr_{g_2g_3}( [T_{s(g_1)}M] \oplus [A^t_{s(g_3)}] )\right)  }
	\end{eqnarray}
	for any particular choice of splittings $\hl,\hr$ of $TG$, defining $\Sl,\Sr$ and $\phih$ as above, and of linear bases $[A^s_{t(g_1)}]$, $[T_{s(g_1)}M]$, $[T_{s(g_2)}M]$, $[A^t_{s(g_2)}]$,  $[A^t_{s(g_3)}]$ of the corresponding spaces.
\end{proposition}

\begin{proof}
Let us first show that \eqref{eq:assocsigma} is equivalent to
\begin{eqnarray}
\frac{\sigma\left(\psi^R_{g_1g_2,g_3}([T_{g_1g_2}G]\times [Ker D_{g_3}t])\right)\sigma\left(\psi^L_{g_1,g_2}([KerD_{g_1}s]\times [T_{g_2}G] ) \right)}{\lambda_G([T_{g_1g_2}G])} =  \nonumber \\
= \frac{\sigma\left(\psi^L_{g_1,g_2g_3}( [Ker D_{g_1}s] \times [T_{g_2g_3}G])\right)\sigma\left(\psi^R_{g_2,g_3}( [T_{g_2}G] \times [KerD_{g_3}t] ) \right)}{\lambda_G([T_{g_2g_3}G])} \label{eq:assocsimpl}
\end{eqnarray}
for any particular choice of splittings $\hl,\hr$ defining $\psi^L_{g_1,g_2}(v_1,v_2)=(v_1+h^L_{g_1}(Tt(v_2)),v_2)$ and $\psi^R_{g_1,g_2}(v_1,v_2)=(v_1,v_2+h^R_{g_2}(Ts(v_1)) )$, with $v_j\in T_{g_j}G$, and for any choice of bases $[Ker D_{g_1}s]$, $[T_{g_2}G]$, $[Ker D_{g_3}t]$, $[T_{g_1g_2}G]$ and $[T_{g_2g_3}G]$.
Given such a choice of bases, denoting $[A^s_g]:=[Ker D_{g}s]$ and $[A^t_g]:=[Ker D_{g}t]$ for simplicity, we set 
$$[T_{(g_1,g_2,g_3)}G^{(3)}]:= (id\times \psi^R_{g_2,g_3})\circ(\psi^L_{g_1,g_2}\times id)([A^{s}_{g_1}]\oplus[T_{g_2}G]\oplus[A^{t}_{g_3}]).$$

Focusing on the l.h.s. of \eqref{eq:assocsigma}, we are in the setting of \eqref{eq:compofs} with $S_1=G^3\supset D_1=G\compo\times G$ and $f_1=m\times id_G$, $S_2=G^2\supset D_2=G\compo$ with $f_2=m$ and $S_3=G$. Note that, in this case, $D_0=G^{(3)}\ni x=(g_1,g_2,g_3)$ and we can take $C_x = \{(0_{g_1},0_{g_2},h^R_{g_3}(v)):v\in T_{tg_3}M\}\subset TG^3$. Moreover, we choose $[T_{f_1(x)}D_2]=\psi^R_{g_1g_2,g_3}([T_{g_1g_2}G] \oplus [A^t_{g_3}])$.
We then compute using relation \eqref{eq:compofs}:
$$\sigma \circ (\sigma \times \lambda_G)([T_{(g_1,g_2,g_3)}G^{(3)}]) = $$
$$=\frac{\sigma([T_{f_1(x)}D_2]) \ (\sigma \times \lambda_G)([T_{(g_1,g_2,g_3)}G^{(3)}]\cup C_x )}{\lambda_{G\times G}([T_{f_1(x)}D_2]\cup (\{0\}\times h_{g_3}^R([T_{t(g_3)}M])))}$$
 $$=\frac{\sigma(\psi^R_{g_1g_2,g_3}([T_{g_1g_2}G]\oplus [A^t_{g_3}]))\ \sigma(\psi^L_{g_1,g_2}([A^{s}_{g_1}]\oplus [T_{g_2}G]))}{\lambda_G([T_{g_1g_2}G])}.$$
In the last equality above, we used that the change of bases taking $[T_{(g_1,g_2,g_3)}G^{(3)}]\cup C_x $ and $[T_{f_1(x)}D_2]\cup \{0\}\times h_{g_3}^R([T_{t(g_3)}M])$ into the product ones $\psi^L_{g_1,g_2}([A^s_{g_1}]+[T_{g_2}G])\times([A^t_{g_3}]\cup h^R_{g_3}[T_{tg_3}M])$ and $[T_{g_1g_2}G]\times ([A^t_{g_3}]\cup h^R_{g_3}[T_{tg_3}M])$, respectively, have determinant $1$.
Computing analogously $\sigma \circ (\lambda_G\times \sigma)([T_{(g_1,g_2,g_3)}G^{(3)}])$, we finish the proof of \eqref{eq:assocsimpl}.

Finally, to get \eqref{eq:assocsimpl2}, we evaluate \eqref{eq:assocsimpl} on the following choice of bases: 
$$[Ker D_{g_1}s] = TR_{g_1}[A^s_{t(g_1)}], \ [T_{g_2}G]= \Sr_{g_2}([T_{s(g_1)}M]\oplus [A^t_{s(g_2)}]),$$
	 $$ \ [Ker D_{g_3}t] =TL_{g_3}[A^t_{s(g_3)}] , \ [T_{g_1g_2}G] = \Sl_{g_1g_2}([A^s_{t(g_1)}] \oplus [T_{s(g_2)}M]), \ \ [T_{g_2g_3}G]=\Sr_{g_2g_3}([T_{s(g_1)}M]\oplus [A^t_{s(g_3)}]),$$
and by direct computation using the identity
$$ \phih_{g_1,g_2}(k^s,v,\tilde k^t) = (\Sl_{g_1}(k^s,v), TL_{g_2}(\tilde k^t) + h^R_{g_2}(Ts(\Sl_{g_1}(k^s,v))) )=(TR_{g_1}(k^s)+h^L_{g_1}(Tt(\Sr_{g_2}(v,\tilde k^t)),  \Sr_{g_2}(v,\tilde k^t)).$$
This finishes the proof.
\end{proof}

The explicit form above immediately implies the following.
\begin{corollary}\label{cor:f}
If there exists an enhancement $\sigma_0$ of $m$ satisfying \eqref{eq:assocsigma} and which is non-vanishing, $\sigma_0\neq 0$, any other solution of \eqref{eq:assocsigma} must be of the form $f\sigma_0$ for a unique smooth $f:G^{(2)} \to \C$ satisfying
\begin{equation}\label{eq:assocf}
f(g_1g_2,g_3) f(g_1,g_2) = f(g_1, g_2g_3) f(g_2,g_3), \ \forall (g_1,g_2,g_3)\in G^{(3)}
\end{equation}
\end{corollary}
Nonvanishing enhancements are thus identified with multiplicative 2-cocycles, as follows.
\begin{remark}($2$-cocycles)\label{rmk:cocycl}
	If we assume further that $f\neq 0$, so that $f:G^{(2)} \to \C^*$ can be seen as a multiplicative $2$-cochain for $G$, then \eqref{eq:assocf} can be interpreted as a \emph{multiplicative $2$-cocycle condition} for $f$: $ \delta_{\C^*} f = 1, $
	where $\delta_{\C^*} f: G^{(3)}\to \C^*$ is defined by
	$$ \delta_{\C^*} f(g_1,g_2,g_3) := f(g_2,g_3)f(g_1g_2,g_3)^{-1}f(g_1,g_2g_3)f(g_1,g_2)^{-1}.$$
	Moreover, when $f = e^{\tilde f}$ for $\tilde f: G^{(2)} \to \C$, then $\tilde f$ has to be an (ordinary) \emph{additive} $2$-cocycle on $G$,
	\begin{equation}\label{eq:addidelta}
	\delta \tilde f = 0, \ \delta \tilde f(g_1,g_2,g_3): = \tilde f(g_2,g_3) - \tilde f(g_1g_2,g_3) + \tilde f(g_1,g_2g_3)- \tilde f(g_1,g_2).
	\end{equation}
\end{remark}

\medskip

\titpar{Canonical enhancements and the main result.}
Moving towards showing existence, let us consider the following exact sequence: for $(g_1,g_2)\in G\compo$,
\begin{equation}\label{eq:shortseq}
0 \to T_{(g_1,g_2)}G^{(2)} \to T_{g_1}G \times T_{g_2}G  \overset{Ds_1-Dt_2}{\to} T_{s(g_1)=t(g_2)}M \to  0
\end{equation}

\begin{definition}\label{def:canonsigma}
Given a non-vanishing half-density on $M$, $\mu \in \Gamma |TM|^{1/2}, \ \mu\neq 0$, and considering the Liouville half density $\lambda_G$ in $(G,\omega)$, the half density $\sigma^c$ on $G^{(2)}$ defined by 
\[ \sigma^c = (\lambda_G\times \lambda_G)/\mu \in \Gamma |TG^{(2)}|^{1/2} \]
via the above short exact sequence is called the \textbf{canonical enhancement} of $\gr(m)$ associated with $\mu$.
\end{definition}
Note that, by construction, the canonical enhancement is non-vanishing, $\sigma^c\neq 0$. As we shall see in \S\ref{subsec:props}, the data of $\mu$ can be seen as an enhancement of the units $1:M\hookrightarrow G$, which is also a Lagrangian submanifold (see the corresponding identity axiom in \eqref{eq:idaxiom}).

\begin{theorem}\label{thm:main} (Existence and classification of enhancements) Let $(G\rightrightarrows M, \omega)$ be a symplectic groupoid. 
	For any choice of non-vanishing half-density $\mu$ on $TM$, then the associated canonical enhancement $\sigma^c$ of $\gr(m)$ satisfies the associativity condition \eqref{eq:assocsigma}. Consequently, $(G\rra M,\omega,\sigma^c)$ defines an enhanced symplectic groupoid and any other associative enhancement of $m$ must be of the form $f\sigma^c$ for a unique function $f:G^{(2)}\to \C$ satisfying equation \eqref{eq:assocf}. 
\end{theorem}

\begin{proof}
Let us first use the splitting \eqref{eq:phih} of $TG\compo$ coming from splittings for $G$ and use it to prove the following characterization of $\sigma^c$:
\begin{eqnarray}\label{eq:defsigc}
\sigma^c|_{g_1,g_2}( \phih_{g_1,g_2}\left( [A^s_{t(g_1)}]\oplus [T_{s(g_1)}M] \oplus [A^t_{s(g_2)}]\right))= \nonumber \\
=\frac{ \lambda_G\left(\Sl_{g_1}([A^s_{t(g_1)}]\oplus [T_{s(g_1)}M])\right) \lambda_G\left(\Sr_{g_2}([T_{s(g_1)}M]\oplus [A^t_{s(g_2)}])\right) }{\mu([T_{s(g_1)}M])}
\end{eqnarray}
for any choice of basis $[A^s_{t(g_1)}], [T_{s(g_1)}M], [A^t_{s(g_2)}]$. The defining short exact sequence \eqref{eq:shortseq} is isomorphic to the following one through the indicated maps,

\hskip-0.5cm
\begin{tikzcd}
0\arrow{r} & A^s_{t(g_2)}\oplus T_{s(g_1)}M\oplus A^t_{s(g_1)}\arrow{d}{\phi^h_{(g_1,g_2)}}\arrow{r}{i'} & A^s_{t(g_2)}\oplus T_{s(g_1)}M\oplus T_{s(g_1)}M\oplus A^t_{s(g_1)}\arrow{r}{P_2-P_1} \arrow{d}{\Sigma^L_{g_1}\times \Sigma^R_{g_2}} & T_{s(g_1)}M\arrow{r} \arrow{d}{id} & 0\\0 \arrow{r} & T_{(g_1,g_2)}G^{(2)} \arrow{r}{i} & T_{g_1}G\oplus T_{g_2}G\arrow{r}{Ds-Dt} & T_{s(g_1)}M\arrow{r} & 0  
\end{tikzcd}

\noindent where $i'(k^s,v,k^t)=(k^s,v,v,k^t)$ and $(P_1-P_2)(k^s,v_1,v_2,k^t)=v_1-v_2$. Eq. \eqref{eq:defsigc} then follows by evaluating the corresponding quotient $(\phi^h_{(g_1,g_2)})^*\sigma^c = ((\Sigma^L_{g_1})^*\lambda_G\times (\Sigma^R_{g_2})^*\lambda_G)/\mu$, as explained below \eqref{eq:sesV}, with $\beta_1= [A^s_{t(g_1)}] + \Delta_{[T_{s(g_1)}M]} + [A^t_{s(g_2)}]$ and $\beta_2 = 0 + (0+ [T_{s(g_1)}M]) + 0$.

Finally, we use Prop. \ref{prop:crit2} to directly verify eq. \eqref{eq:assocsimpl2} for $\sigma=\sigma^c$ characterized by \eqref{eq:defsigc}. The only non-obvious step consists in evaluating the following factor appearing in the r.h.s. of \eqref{eq:assocsimpl2},
	$$ (\star) = \sigma^c|_{g_2,g_3} \phih_{g_2,g_3} \left( (\Sl_{g_2})^{-1}\Sr_{g_2}\left( [T_{s(g_1)}M] \oplus [A^t_{s(g_2)}] \right) \oplus [A^t_{s(g_3)}] \right) $$
	since the basis in which $\phih$ is evaluated above is not in the form present in the formula \eqref{eq:defsigc} for $\sigma^c$. Notice that
	$$ (\Sl_{g_2})^{-1}\Sr_{g_2} : T_{s(g_1)}M \oplus A^t_{s(g_2)} \to A^s_{s(g_1)}\oplus T_{s(g_2)}M.$$
	Let $[A^s_{s(g_1)}]$ be an arbitrary basis and denote $B$ the change of basis
	$$ (\Sl_{g_2})^{-1}\Sr_{g_2}\left( [T_{s(g_1)}M] \oplus [A^t_{s(g_2)}] \right)= ([A^s_{s(g_1)}] \oplus [T_{s(g_2)}M])\cdot B,$$
	where $[T_{s(g_2)}M]$ is the one given in \eqref{eq:assocsimpl2}. Then,
	\begin{eqnarray*} (\star)& = |det(B)|^{1/2} \sigma^c|_{g_2,g_3} \phih_{g_2,g_3} ( [A^s_{s(g_1)}] \oplus [T_{s(g_2)}M] \oplus [A^t_{s(g_3)}] ) \\
	& = |det(B)|^{1/2} \frac{ \lambda_G\left(\Sl_{g_2}([A^s_{s(g_1)}]\oplus [T_{s(g_2)}M])\right) \lambda_G\left(\Sr_{g_3}([T_{s(g_2)}M]\oplus [A^t_{s(g_3)}])\right) }{\mu([T_{s(g_2)}M])} \\
		& =  \frac{ \lambda_G\left(\Sl_{g_2}\left( ([A^s_{s(g_1)}]\oplus [T_{s(g_2)}M])\cdot B \right)\right) \lambda_G\left(\Sr_{g_3}([T_{s(g_2)}M]\oplus [A^t_{s(g_3)}])\right) }{\mu([T_{s(g_2)}M])} \\
& =  \frac{ \lambda_G\left(\Sr_{g_2}\left( [T_{s(g_1)}M] \oplus [A^t_{s(g_2)}] \right)\right) \lambda_G\left(\Sr_{g_3}([T_{s(g_2)}M]\oplus [A^t_{s(g_3)}])\right) }{\mu([T_{s(g_2)}M])}.
	\end{eqnarray*}
With this identity, the proof follows.
\end{proof}
\begin{remark} (Understanding by analogy the canonical solution)
	There is a heuristic structural way of understanding why formula \eqref{eq:defsigc} in the proof solves the associativity equation. In the spirit of Remark \ref{rmk:cocycl}, translating multiplicative cocycle conditions into additive cocycle conditions, the associativity condition for a $\sigma$ in its form \eqref{eq:assocsimpl} has the following analogous structure: find $\tilde f:G^{(2)}\to \C$ so that 
	$$ \delta \tilde f (g_1,g_2,g_3) = F(g_2g_3) - F(g_1g_2) $$
	for a given $F:G \to \C$. In the analogy, $\sigma$ plays the role of $\tilde f$ and the Liouville $\lambda_G$ the role of the given $F$ above. The interesting point is that 
	$ \tilde f(g_1,g_2) := F(g_1) + F(g_2) - h(s(g_1))$
	is always a solution of the above equation, for any $h:M \to \C$. Recalling that $F$ represents $\lambda_G$ in the analogy, taking $h$ to play the role of the half-density $\mu$ on $TM$, and switching from additive to multiplicative, we obtain precisely the structure of the solution $\sigma^c$ in \eqref{eq:defsigc}. 
\end{remark}

\titpar{Homological interpretation of the classification.}
Besides multiplicative cocycles recalled in Remark \ref{rmk:cocycl}, we also specialize the above result to a sub-class of enhancements which will be important in semi-classical limits of concrete star products, Section \ref{sec:app}.
\begin{definition}\label{def:exptype}
Consider a symplectic groupoid $(G\rra M,\omega)$ and a non-vanishing half-density $\mu$ in $M$ with associated canonical enhancement $\sigma^c$. We say that an enhancement $\sigma$ is of \textbf{exponential type} relative to $\mu$ when 
\[ \sigma = e^h \sigma^c, \ h:G\compo \to \C. \]
We say that two such exponential enhancements $\sigma_1, \ \sigma_2$ are \textbf{exp-equivalent} if they are equivalent through an automorphism with $\kappa = e^{\tilde h}$ in Definition \ref{def:equivenhanc}, for some $\tilde h: G \to \C$.
\end{definition}

We can then characterize the two sets of possible non-vanishing enhancements modulo equivalence.
\begin{corollary}\label{cor:H2}
 Consider a symplectic groupoid $(G\rra M,\omega)$. Each choice of non-vanishing half density $\mu$ on $M$ induces an identification
\[\{\text{equiv. classes }[\sigma]: \text{$\sigma$ non-vanishing enhancement of $(G,\omega)$ } \} \simeq  H^2(G,\mathbb{C^*}), \] 
 between the equivalence classes of Definition \ref{def:equivenhanc} and the second multiplicative differentiable cohomology group $H^2(G,\mathbb{C^*})$. Additionally, the set of exponential enhancements relative to $\mu$ modulo exp-equivalence is in bijection with the second additive cohomology group $H^2(G)$ of $G\rra M$.    
\end{corollary}
Notice that the above corollary can also be seen as providing an interpretation for the cohomology group $H^2(G,\mathbb{C^*})$: for  a symplectic groupoid $(G\rra M,\omega)$, classes $[f]\in H^2(G,\mathbb{C^*})$ can be interpreted as providing \emph{non-vanishing associative deformations} $f\sigma^c$ of a canonical enhancement $\sigma^c$ modulo the equivalences of Definition \ref{def:equivenhanc}.

\begin{remark} \label{rmk:defo}(The underlying deformation class of $(M,\pi)$)
Recall the van Est map $v: H^k(G) \to H^k(A)$ from differentiable cohomology for $G\rra M$ to Lie algebroid cohomology for $A=Lie(G)$ (see \cite{Cra03},\cite{LB-M}, and also \S\ref{subsec:props} below). For a symplectic groupoid, $A\simeq T^*_\pi M$ is the cotangent Lie algebroid associated with the underlying Poisson manifold $(M,\pi)$. Then, given an exponential type enhancement $\sigma = e^h \sigma^c$, we thus obtain a Lie algebroid class
\[ v([h]) \in H^2(T^*_\pi M). \]
Such cohomology elements can be interpreted as classes $[\pi_1]$ of first order deformation parameters of the Poisson structure $\pi_t = \pi + t \pi_1 + O(t^2)$ modulo trivial ones. We then conclude that exponential enhancements modulo exp-equivalence define an underlying deformation class $[\pi_1]=v([h])$ for $(M,\pi)$. When $\sigma$ comes from a star product with underlying Kontsevich class $\h\pi_0 + \h^2 \pi_1 +\dots$ (see \cite{Kont}), it is expected that the class $v([h])$ above corresponds to the first correction term $\pi_1$, this will be explored elsewhere.
\end{remark}

\subsection{Properties and simple examples}
\label{subsec:props}
We first prove some immediate properties of enhanced symplectic groupoids and then provide a list of illustrative examples.

\titpar{Properties: non-vanishing and identity axiom.}
The first property concerns the non-vanishing property, for which we recall the notation $1\compo_x = (1_x,1_x)\in G\compo$ for $x\in M$.
\begin{lemma}\label{lem:sigmanonvanish}
Let $\sigma$ be an associative enhancement of a symplectic groupoid $(G\rra M,\omega)$. If $\sigma|_{1\compo_M}\neq 0$ and $G$ is source-connected, then $\sigma\neq 0$ is non-vanishing everywhere in $G\compo$.
\end{lemma}

\begin{proof}
Using Corollary \ref{cor:f} and Theorem \ref{thm:main}, we can write $\sigma=f \sigma^c$ with $f:G\compo \to \C$ satisfying \eqref{eq:assocf}. In this setting, we want to show that the condition
\[ (a): f(1_x,1_x)\neq 0,\ \forall x\in M\] 
implies $f(g_1,g_2)\neq 0$ for any $(g_1,g_2)\in G\compo$.

Let $x\in M$ and $g\in G$ with $t(g)=x$. The associativity condition \eqref{eq:assocf} for $f$, with $g_1=g_2=1_x$ and $g_3=g$, implies
$f(1_x,1_x) f(1_x,g) = f(1_x,g) f(1_x,g)$.
Thus, $f(1_x,g)$ can only take the values $f(1_x,1_x)$ or zero. Since for $g=1_x$ we have $f(1_x,1_x)\neq 0$ by the hypothesis (a), and since $G$ is $s$-connected and $f$ is continuous, we conclude
\[(b): f(1_x, g) = f(1_x,1_x), \forall g\in t^{-1}(x). \]
Analogously, $f(g,1_x)=f(1_x,1_x) \forall g \in s^{-1}(x)$.
Finally, we fix $g_2$ in $t^{-1}(x)$ and vary $g_1\in s^{-1}(x)$. We want to show that the set $V_{g_2} \subset s^{-1}(x)$ defined by $f(g_1,g_2)=0$ is both closed and open. Since the source fiber is connected and $f(1_x,g_2)=f(1_x,1_x)\neq 0$ by (a,b) above, we shall conclude that $V_{g_1}$ is empty, thus concluding the proof.
The fact that $V_{g_2}$ is closed is obvious since $f$ is continuous. To show that it is also open we observe that, when $g_1'\in s^{-1}(x'), \ x'=t(g_1)$, is close enough to $1_{x'}$, by continuity w.r.t. (a,b) above we get $f(g_1',g_1)\neq 0$ and $f(g_1',g_1g_2)\neq 0$. Lastly, using the associativity condition \eqref{eq:assocf} for $f$ we get
\[f(g_1',g_1)f(g_1'g_1,g_2) = f(g_1', g_1g_2) f(g_1,g_2) \]
so that $f(g_1,g_2)=0$ iff $f(g_1'g_1,g_2)=0$ for every $g_1' \in s^{-1}(x')$ close enough to $1_{x'}$. This finishes the proof. 
\end{proof}

The second property concerns an underlying \emph{identity axiom} for an enhanced symplectic groupoid $(G\rra M,\omega,\sigma)$. 
Given a half-density $\mu$ on $M$, we say that the morphism 
$(1_M, \mu):\ast \dto (G,\omega)$
\textbf{satisfies the identity axiom} for $(G\rra M,\omega,\sigma)$ if the following compositions yield the identity morphism $\id_G=(\gr(id_G),\lambda_G)$ on $(G,\omega)$:
\begin{eqnarray}
G\times \ast \overset{\id_G\times (1_M,\mu)}{\dto} G\times G \overset{(\gr(m),\sigma)}{\dto} G, \ \ (\gr(m),\sigma)\circ(id\times (1_M,\mu))=\id_G \nonumber \\  
 \ast \times G \overset{ (1_M,\mu) \times \id_G}{\dto} G\times G \overset{(\gr(m),\sigma)}{\dto} G, \ \ (\gr(m),\sigma)\circ((1_M,\mu)\times id)=\id_G.\label{eq:idaxiom}
\end{eqnarray}
Note that $G\times \ast = G = \ast \times G$ and, thus, it makes sense to describe the Lagrangian $\gr(id_G)$ as a canonical relations $G\times \ast = \ast \times G \dto G$.

We want to show that, when $\sigma$ is non-vanishing, there is always such a $\mu$ satisfying the identity axiom. Let us consider $\sigma \neq 0$ and the induced half-density $\mu_\sigma$ on $M$ through the exact sequence \eqref{eq:shortseq} with $g_1=1_x=g_2, \ x\in M$,
$$\mu_\sigma|_x = (\lambda_G|_{1_x} \times \lambda_G|_{1_x})/\sigma|_{(1_x,1_x)}.$$
Note that it satisfies the scaling property $\mu_{f\sigma} = (1/f|_{1\compo_M}) \mu_\sigma$ for $f:G\compo\to \C$ non-vanishing. Also notice that, if $\sigma^c$ is a canonical enhancement associated with a given $\mu\neq 0$ on $M$ as in Definition \ref{def:canonsigma}, then the induced half-density recovers $\mu$,  $\mu_{\sigma^c} = \mu$.
We can now prove the identity axiom for $\mu_\sigma$.
\begin{proposition}\label{prop:idaxiom}
Let $(G\rra M,\omega,\sigma)$ be an enhanced symplectic groupoid with non-vanishing $\sigma$, and let $\mu_\sigma$ be the half-density on $M$ defined above. Then, the enhanced morphism $(1_M,\mu_\sigma):\ast \dto (G,\omega)$
satisfies the identity axiom of eq. \eqref{eq:idaxiom}.
\end{proposition}

\begin{proof}
Let us first assume that there is a non-vanishing associative enhancement $\sigma_0$ of $(G\rra M, \omega)$ such that $\mu_{\sigma_0}$ satisfies the identity axiom on $(G,\omega,\sigma_0)$ and show that, for any other associative $\sigma = f\sigma_0$ with $f:G\compo \to \C$ non-vanishing, the induced $\mu_\sigma$ also satisfies the identity axiom on $(G,\omega,\sigma)$. After this, we shall show that a canonical enhancement $\sigma_0=\sigma^c$ has the above property, thus completing the proof.

For $\sigma=f\sigma_0$ an associative enhancement as above, we have for $g\in G$ with $s(g)=x$,
\[ (\sigma \circ (\lambda_G\times \mu_\sigma) )|_g = \frac{f(g,1_{x})}{f(1_{x},1_{x})} (\sigma_0 \circ (\lambda_G \times \mu_{\sigma_0}) )|_g = \frac{f(g,1_{x})}{f(1_{x},1_{x})} \lambda_G|_g,\]
where we have used the scaling property for $\mu_{f\sigma}$ mentioned above the Proposition and the hypothesis on $\sigma_0$. Similarly to the argument given in the proof of Lemma \ref{lem:sigmanonvanish}, the associativity condition \eqref{eq:assocf} for $f$ implies that $f(g,1_x)$ can only be $f(1_x,1_x)$ or zero. The non-vanishing condition on $f$ then implies that $f(g,1_x)=f(1_x,1_x)$ so that the above composition yields $\lambda_G$. The case of $\sigma \circ (\mu_\sigma\times \lambda_G)=\lambda_G$ is similar, thus showing that, in this case, $\mu_\sigma$ satisfies the identity axiom on $(G,\omega,\sigma)$.

We are thus left with showing 
that, for any half-density $\mu\neq 0$ on $M$, the associated canonical enhancement $\sigma_0=\sigma^c$ of Definition \ref{def:canonsigma} is such that $\mu_{\sigma^c}=\mu$ satisfies the identity axiom. This follows directly from the composition formula in Remark \ref{rmk:linearcompoformula}, the details are left to the reader.
\end{proof}


\begin{remark} (An identitiy involving the inverse) Consider an enhancement of the form $\sigma = f \sigma^c$ so that $f$ satisfies the associativity condition \eqref{eq:assocf}. Using the arguments in the proof of Lemma \ref{lem:sigmanonvanish}, it follows that, assuming $f|_{1\compo_M}\neq 0$,
\[ f(g,g^{-1}) = \frac{f(1_{s(g)},1_{s(g)})}{f(1_{t(g)},1_{t(g)})} f(g^{-1},g). \]
This identity makes its appearance in the study of non-formal star products (see \cite{CFer2}).
\end{remark}

\titpar{Examples.} 
Next, we illustrate the definition and the results of this chapter in concrete examples. We shall be using a well-known cohomological result that we now recall. Given a Lie groupoid $G\rra M$ with Lie algebroid $A\to M$, there is an induced \emph{van Est map} from differentiable cohomology to Lie algebroid cohomology,
\begin{equation}\label{eq:vanEst}
v: H^k(G) \to H^k(A).
\end{equation}
The corresponding \emph{van Est theorem} (\cite{Cra03}, see also \cite{LB-M}) says that this map is an isomorphism for all $k\leq n$ when the source fibers of $G$ are $n$-connected. We also recall that, for a symplectic groupoid $(G\rra M,\omega)$, there is a natural isomorphism $A\simeq T^*_\pi M$ to the cotangent algebroid of the underlying Poisson $(M,\pi)$ (see \cite{Macbook}).

The first example shows that enhancements can be non-unique and non-trivial even for the simplest symplectic groupoid, which also plays a role in Section \ref{sec:app} as the point around which (formal) deformations are taken within Kontsevich's quantization formalism.

\begin{example}\label{ex:pi=0} (Enhanced groupoids for $\pi=0$)
Let then $M$ any manifold endowed with the trivial Poisson structure, $\pi=0$. The corresponding source 1-connected symplectic groupoid $(G\rra M, \omega)$ is given by
\[ G=T^*M, \ s=t=q: T^*M \to M\text{ bundle projection}, \ \omega=\omega_c \text{ canonical,} \]
where we denote $g=\alpha \in T^*M$
and the groupoid multiplication is
$m(\alpha_1,\alpha_2) = \alpha_1 + \alpha_2, \ \alpha_1,\alpha_2 \in T^*_x M$.
Given $\mu$ a non-vanishing half-density on $M$, the corresponding canonical enhancement $\sigma^c$ can be described as follows. 
Using a linear connection on $T^*M\to M$, we can split any tangent space into horizontal and vertical parts
$\phi: q^*(TM)\times_{T^*M} q^*(T^*M) \to T(T^*M)=TG$
so that, for each $x\in M$, $\phi^*\lambda_{T^*_xM}=\mu_x \otimes \tilde\mu_x$ with $\tilde \mu$ the dual of $\mu$. In canonical coordinates $(x^i,p_j)$ for $T^*M$, if we take $\mu=|dx|^{1/2}$ then $\tilde \mu=|dp|^{1/2}$. The identification $\phi$ also induces
$$\hat \phi: T_x M \times T_x^*M \times T_x^* M \simeq T_{(\alpha_1,\alpha_2)}(T^*M\times_M T^*M)=T_{g_1,g_2}G\compo $$
so that, from the description of $\sigma^c$ in Definition \ref{def:canonsigma}, we get $\hat\phi^*\sigma^c = \mu \otimes \tilde \mu \otimes \tilde \mu$. If we take $\sigma = f \sigma^c$ for $f:T^*M\times_M T^*M\to \C$, we get by direct computation
\[ \sigma\circ(\sigma \times \lambda_G)|_{(g_1,g_2,g_3)} \simeq f(\alpha_1,\alpha_2)f(\alpha_1+\alpha_2,\alpha_3) \ \mu \otimes \tilde \mu \otimes \tilde \mu \otimes \tilde \mu, \]
\[ \sigma\circ(\lambda_G \times \sigma)|_{(g_1,g_2,g_3)} \simeq f(\alpha_2,\alpha_3)f(\alpha_1,\alpha_2+\alpha_3) \ \mu \otimes \tilde \mu \otimes \tilde \mu \otimes \tilde \mu, \]
using the induced $TM\times_M T^*M\times_M T^*M\times_M T^*M\simeq TG^{(3)}$.
This recovers a specialization of eq. \eqref{eq:assocf} for $f$,
\[f:T^*M\times_M T^*M \to \C, \ f(\alpha_1,\alpha_2)f(\alpha_1+\alpha_2,\alpha_3)= f(\alpha_2,\alpha_3)f(\alpha_1,\alpha_2+\alpha_3). \]
Constants $f=c$ solve this equation but there are also non-constant solutions e.g. $f=e^{h}$ for any additive $2$-cocycle $h:G\compo=T^*M\times_M T^*M \to \C$, $\delta h=0$. Using the van Est isomorphism, the exp-equivalence classes of such solutions are in bijections with bivectors on $M$, $v([h])\in H^2(A)\simeq \mathfrak{X}^2(M)$ (since $d=0$ on $A=T^*_{\pi=0}M$ in this case). We thus see that, even in this simple case, we have an infinite dimensional space of non-equivalent associative enhancements.
\end{example}

\begin{example}\label{ex:symplectic} (Enhanced groupoids for $\pi=\omega_M^{-1}$ symplectic)
Let $(M,\omega_M)$ be a symplectic manifold of dimension $dim(M)=2d$, seen as a Poisson manifold with non-degenerate $\pi=\omega_M^{-1}$.
We consider the symplectic groupoid $(G=M\times M \rra M,\omega)$ given by the pair groupoid $M\times M\rra M$ endowed with $\omega = p_1^*\omega_M - p_2^*\omega_M$, where the target is $p_1(x,y)=x$ and the source is $p_2(x,y)=y$. It follows that the Liouville half-density on $G=M\times M$ is $\lambda_G = \lambda_M \otimes \lambda_M$ (up to constant).
In this case, we have a natural non-vanishing half-density  on $M$, $\mu=\lambda_M$. The corresponding canonical enhancement is
\[ \sigma^c = \lambda_M \otimes \lambda_M \otimes \lambda_M\text{ on } G\compo=M^3. \]
Any other associative enhancement will be of the form $\sigma = f \sigma^c$ for $f:G\compo=M^3 \to \C$
satisfying
\[ f(x_1,x_2,x_3) f(x_1,x_3,x_4) = f(x_2,x_3,x_4) f(x_1,x_2,x_4). \]
We remark that, as in the proof of Lemma \ref{lem:sigmanonvanish}, it follows that $f(x,y,y)$ (resp. $f(x,x,y)$) can only take the values zero or $f(y,y,y)$ (resp. $f(x,x,x)$). Finally, let us focus on exponential enhancements $f=e^h$ with $h:M^3 \to \C$ (Definition \ref{def:exptype}), whose exp-equivalence classes correspond to 2-cocycles $[h]\in H^2(G)$ for the pair groupoid. When $M$ is $2$-connected, using the van Est map \eqref{eq:vanEst}, they are then in bijection with ordinary de Rham cohomology classes
\[ [v(h)] \in H^2_{dR}(M), \]
since $A\simeq TM$ the tangent algebroid in this case.
Using Remark \ref{rmk:coverings}, an analogous description also holds for the source 1-connected symplectic groupoid $(\Pi_1(M)\rra M, \omega')$ given by the fundamental groupoid.
We also observe that non-canonical enhancements $\sigma\neq \sigma^c$ appear as semiclassical limits of integral quantization formulas for $(M,\omega_M)$ from Jacobian-type factors, see e.g. \cite[eq. (16)]{Biel} and \cite[eq. (5.16)]{KarOsb}. These yield another source of non-trivial examples of enhanced symplectic groupoids.
\end{example}

We observe that when $M = \R^n$ and $\pi$ is constant, we can do a global Weinstein splitting $M\simeq M_1 \times M_2$ with $M_1\simeq Im(\pi^\sharp)$ constant symplectic and $M_2$ endowed with $\pi_2=0$. The corresponding enhanced symplectic groupoids are products of the ones described in the previous two examples.

The final example will be relevant in the study of star products in Section \ref{sec:app}.
\begin{example}\label{ex:pilinear} (Enhanced groupoids for $\pi$ linear)
Let $\g$ be a (finite dimensional, real) Lie algebra with bracket $[,]$, $dim(\g)=n$, and consider $M=\g^*$ endowed with the linear Poisson structure given by $\pi^{ij}(x) = -c^{ij}_k x^k$, where $e^i$ defines a basis of $\g$ with dual basis $e_i$ of $\g^*$ with respect to which $x^i$ are linear coordinates and $[e^i,e^j]=c^{ij}_k e^k$. (The minus sign in $\pi$ is conventional.)

For any Lie group $\G$ integrating $\g$, it is known (see e.g. \cite{CDW,Macbook}) that the cotangent bundle inherits a natural "cotangent lift" Lie groupoid structure $T^*\G \rra \g^*$ such that $(T^*\G\rra \g^*, \omega_c)$ integrates $(M=\g^*,\pi)$. This will be described further in Section \ref{sec:app}. Moreover, left translations on $\G$ induce an isomorphism $(T^*\G\rra \g^*,\omega_c)\simeq (G_\ltimes=\G\ltimes \g^*\rra \g^*,\omega)$ onto the action groupoid associated with the coadjoint action of $\G$ on $\g^*$ and endowed with 
\[\omega =d(\langle x, \theta\rangle) =\langle dx \overset{\wedge}{,} \theta\rangle - \frac{1}{2} \langle x, [\theta\overset{\wedge}{,}\theta]\rangle ,\]
where $\langle,\rangle$ denotes the pairing between $\g^*$ and $\g$, and $\theta:T\G\to \g$ is the left-invariant Maurer-Cartan form. The structure maps in $G_\ltimes$ are
\[ s(g,x) = x, \ t(g,x)= Ad^*_g x, \ 1_x = (e,x),
m((g_1,x_1)(g_2,x_2)) = (g_1g_2, x_2). \]
Let us consider $\mu=|dx|^{1/2}$ the Euclidean half-density on $\g^*$ with dual $\tilde \mu=|dp|^{1/2}$ on $\g$. Using $G_\ltimes\compo\simeq \G\times\G\times \g^*$, the corresponding canonical enhancement of $G_\ltimes$ is 
\[\sigma^c = \tilde\mu_L \otimes \tilde\mu_L \otimes \mu,\text{ for }\tilde \mu_L|_g := L_{g^{-1}}^*\tilde \mu. \]
Other enhancements $\sigma=f\sigma^c$ are determined by
\[f: \G\times \G \times \g^* \to \C: \ f(g_1,g_2,Ad^*_{g_3}x_3)f(g_1g_2,g_3,x_3)=f(g_2,g_3,x_3)f(g_1,g_2g_3,x_3). \]
Focusing on exponential type enhancements $f=e^h$ modulo exp-equivalence, we obtain classes $[h]\in H^2(G_\ltimes)$. When $\G$ is 1-connected (so that it is also $2$-connected for being a Lie group), the van Est map establishes a bijection with
\[ [v(h)]\in H^2(\g,C^\infty(\g^*)), \]
namely, the second Lie algebra cohomology group for $\g$ with values in the $ad^*$-module $C^\infty(\g^*)$. This follows since the algebroid is $A=\g\ltimes \g^*$ in this case.
Note that scalar Lie algebra 2-cocycles $\Lambda ^2 \g \to \C$ can be seen as particular solutions with values in constant functions in $C^\infty(\g^*)$.
\end{example}

\section{Application: complete semiclassical factors in Kontsevich's star product}
\label{sec:app}
In this section, we go back to the starting motivation and apply the general theory to the study of star products. We focus on {\bf coordinate Poisson manifolds}, namely, $M\simeq \R^n$ endowed with an arbitrary Poisson structure $\pi=\frac{1}{2}\pi^{ij}(x) \partial_{x^i}\wedge \partial_{x^j}$. First, in \S\ref{subsec:coordstar} we describe the possible enhancements of the underlying \emph{local} symplectic groupoid of \cite{C22} through explicit formulas. Second, in \S\ref{subsec:Khalf} we show the main result of this section stating that Kontsevich's enhancement is equivalent to the canonical one (Theorem \ref{thm:Kenhancement}). Finally, in \S\ref{subsec:linear} we apply the theory to a linear Poisson structure leading to the special factors behind the Duflo isomorphism (Proposition \ref{prop:gastsigmaK}).

\subsection{Star products and enhanced symplectic groupoids for coordinate Poisson manifolds}\label{subsec:coordstar}

%
Within this subsection, we shall be working with \emph{local symplectic groupoid structures} defined on neighborhoods $G\subset T^*M$ of $0_M$, see \cite{C22} for a detailed relevant setting and their relation to star products (see also \cite{CatDheWei}). Enhanced \emph{local} symplectic groupoids are defined analogously to Definition \ref{def:main} and they enjoy the same properties on the relevant domains: for example, the associativity condition needs to hold only on a neighborhood $U_A \subset G^{(3)}$ of $1^{(3)}_M$. Since most of the considerations involving enhancements are pointwise, the proof of the relevant results carries naturally onto the local-groupoid case.

\titpar{From a $\star$-product to $(S,a_0)$ and associativity conditions}
Following the general description of Fourier Integral Operators (FIO) in \cite{GSbook,Mein}, let us consider, as a general motivation, star products on $M$ given as (see also \cite{CatDheWei,CFer1,CFer2})
\begin{equation}\label{eq:FIOstar}
f_1\star_h f_2|_{x_3} = (2\pi \h)^{-2n} \int_{x_1,x_2,p_1,p_2}f(x_1)f(x_2)a_\h(p_1,p_2,x_3)e^{\fih\left(-p_1x_1-p_2x_2 + S(p_1,p_2,x_3)\right)}
\end{equation}
where $a_\h = \sum_{n\geq 0} \h^n a_n$.
(See \cite{C22,CFer2,CatDheWei} for more details.) A key point is that computing $e^{\fih p_1} \star_\h e^{\fih p_2}$ we formally obtain the type of expansion \eqref{eq:kontstarfactors} recalled in the Introduction.
The associativity condition (S3) for such a $\star_h$ formally implies through a stationary phase approximation the following identities:
\begin{itemize}
\item (\cite{CatDheFel}) the \emph{generating function} $S\equiv S(p_1,p_2,x)$ satisfies the \emph{Symplectic groupoid associativity equation (SGA equation)},
\begin{equation}\label{eq:SGA}
S(p_1,p_2,\bar{x})+S(\bar{p},p_3,x)-\bar{x}\bar{p}=S(p_2,p_3,\Tilde{x})+S(p_1,\Tilde{p},x)-\Tilde{p}\Tilde{x}
\end{equation}
where $\bar{x}=\nabla_{p_1}S(\bar{p},p_3,x)$, $\bar{p}=\nabla_{x}S(p_1,p_2,\bar{x})$,
$\Tilde{x}=\nabla_{p_2}S(p_1,\Tilde{p},x)$ and $\Tilde{p}=\nabla_{x}S(p_2,p_3,\Tilde{x})$;

\item the \emph{leading symbol} $a_0\equiv a_0(p_1,p_2,x)$ satisfies
\begin{eqnarray}\nonumber
    a_0(p_1,p_2,\bar{x})a_0(\bar{p},p_3,x)\left|det(I - \nabla^2_xS_{(p_1,p_2,\bar{x})}\nabla^2_{p_1}S_{(\bar{p},p_3,x)})\right|^{-1/2}= \\
=a_0(p_2,p_3,\Tilde{x})a_0(p_1,\Tilde{p},x)\left|det(I - \nabla^2_xS_{(p_2,p_3,\Tilde{x})}\nabla^2_{p_2}S_{(p_1,\Tilde{p},x)})\right|^{-1/2} \label{eq:Ea0}
\end{eqnarray}   
\end{itemize}

Following \cite{GSbook,Mein} further, the data $(S,a_0)$ can be understood geometrically as an enhanced canonical relation $(L,\sigma)$, as follows. The canonical relation $L:T^*M\times T^*M \dto T^*M$ is (up to sign change in the domain's momenta) the Lagrangian generated by the function 
$$ \phi = -x_1p_1-x_2p_2+S(p_1,p_2,x_3)$$
when reduced along the projection $M^3 \times (M^*)^2 \to M^3, (x_1,x_2,x_3,p_1,p_2)\mapsto (x_1,x_2,x_3)$, see \cite[\S 5]{GSbook}. The general idea is that $L=\gr(m)$ defines the graph of multiplication on an underlying local groupoid structure $G$ on $T^*M$ (see the general theory in \cite{CFer2}). In this context, we say that $S$ is a {\bf (coordinate) generating function} for $G$, see \cite{C22}. The half-density $\sigma=\sigma^{a_0}$ on $L$ is defined by $a_0$ following the general prescription of \cite[\S 8.5]{GSbook} (see also \cite{Mein}). We shall see below the corresponding explicit formula specialized to our case of interest.

\begin{remark}(Non-vanishing)
The axiom (S1) for $\star$ implies that $a_0(x,0,0)\neq 0$. This translates into $\sigma|_{1\compo_M}\neq 0$ on the underlying local groupoid $G$. We can thus assume that $\sigma$ is non-vanishing when considering the germ of $G$ around the units.
\end{remark}

\titpar{A construction of $G_\pi$ and its generating $S_\pi$}
Given any $\pi$ on $M\simeq \R^n$, we shall recall from \cite[\S 3.3, \S 3.4]{C22} the construction of a local symplectic groupoid structure $G_\pi$ on $T^*M=M\times M^*$ and of a corresponding generating function $S_\pi$. The motivation is that, following \cite{C22} further, $S_\pi$ yields the factor $S_K$ to be used in Section \ref{subsec:Khalf}  upon asymptotic expansion.

We follow the conventions of \cite{CabMarSal1} for local Lie groupoids in which each structure map has a domain of definition and each axiom has a domain where is holds. The structure of $G_\pi$ lives on arrows $g=(x,p)\in T^*M=M\times M^*$ which are "small" in the sense $p\sim 0\in M^*$. The symplectic structure is the canonical one and the identities are given by the zero section,
\[ \omega = \omega_c, \ 1_x = (x,0)\in T^*M .\]
The source map $s: U_s \subset T^*M \to M$
is defined on a neighborhood of the zero section via the implicit relation,
\[ Q(s(x,p),p)=x\ \forall p\sim 0, \text{ with } Q(\tx,p):=\int_0^1 \varphi^p_u (\tx) \ du,\]
where $\varphi^p_u: M \to M$ is the time-$u$ flow of the "flat Poisson spray" equation for $x(u)\in M$, $\dot x^i = \pi^{ij}(x) p_j$
with $p\in M^*$ seen as a fixed parameter. The above equation indeed defines a smooth map $(x,p)\mapsto s(x,p)$ via the implicit function theorem, using $s(x,0)=x$. Moreover, the fact that this $s$ defines a symplectic realization $s:(U_s,\omega_c)\to (M,\pi)$ goes back to Karasev, \cite{Karasev}. The inverse map is $inv(x,p)=(x,-p)$ and, thus, the target is $t(x,p)=s(x,-p)$. We analogously have the relation
\[ \tilde Q(t(x,p),p) = x, \forall p\sim 0\text{ with } \tilde Q(\tx,p)=Q(\tx,-p). \]

The rest of the local groupoid structure can be determined by the (strict) symplectic realization data $(U_s,\omega_c,1,s)$, as recalled in \cite[\S 2.2]{C22} from \cite{CDW}. Moreover,
%
%
%
%
following \cite[\S 3.2 and Thm. 3.29]{C22}, we notice that all the structure maps of $G_\pi\rra M$ can be encoded into a single {\bf canonical generating  function}
\[ S\equiv S_\pi : U_S \subset  M^* \times M^* \times M \to \R, (p_1,p_2,x) \mapsto S(p_1,p_2,x) \]
where $U_S$ is a neighborhood of $0 \times 0 \times M$ and, moreover, $S$ admits a description throught the explicit formula \cite[eq. (31)]{C22}. We will only need to recall that the key relation between the function $S$ and the local groupoid structure is given at the level of the graph of the multiplication map,
\begin{eqnarray}
T^*M^3 \supset \gr(m)=\{ ((x_1=\partial_{p_1}S|_{(p_1,p_2,x)},p_1),(x_2=\partial_{p_2}S|_{(p_1,p_2,x)},p_2), (x_3=x, p_3=\partial_xS|_{(p_1,p_2,x)} )) \nonumber \\
 : (p_1,p_2,x)\in U_S \}.\label{eq:Sgrm}
 \end{eqnarray}
This relation actually defines uniquely the germ of the local groupoid structure in terms of $S$. For example, the source and target maps are given by $s(x,p) = \partial_{p_2} S(p,0,x)$ and $t(x,p) = \partial_{p_1}S (0,p,x)$. It also follows that
$$ Q(x,p) = \partial_{p_2}S(-p,p,x), \tilde Q(x,p) = \partial_{p_1}S(p,-p,x).$$
Conversely, following \cite{CatDheFel}, given a function $S$ as above, the structure maps it induces define a local symplectic groupoid structure on $(T^*M,\omega_c)$ if $S$ statisfies the SGA equation \eqref{eq:SGA}.
%

\titpar{Enhancements for $G_\pi$}
We now provide formulas for enhancements of the local symplectic groupoid $G_\pi$ and relate them to the $a_0$ factor coming from a star product of the form \eqref{eq:FIOstar} for $(M,\pi)$.

We first describe the canonical enhancement $\sigma^c$ associated with the euclidean half-density $\mu=|dx|^{1/2}$ on $M\simeq \R^n$. We denote $\tilde \mu=|dp|^{1/2}$ the dual enhancement on $M^*$ so that $\lambda_{T^*M}=\mu \otimes \tilde \mu$. Consider the parametrization of composable arrows in $G_\pi$ given by
\begin{equation}\label{eq:defJ}
J:U_J \subset M^*\times M^* \times M \to G_\pi\compo, (p_1,p_2,x) \mapsto (g_1=(\partial_{p_1}S|_{(p_1,p_2,x)},p_1), g_2=((\partial_{p_2}S|_{(p_1,p_2,x)},p_2) ), \ 
\end{equation}
defined on a neighborhood $U_J$ of $p_1=p_2=0$. In this parametrization, the multiplication map yields
\[ m(J(p_1,p_2,x)) = (x,\partial_xS|_{(p_1,p_2,x)})=:g_3, \text{ 
so that $(g_1,g_2,g_3)\in \gr(m)$.}\]
We denote $\tilde x(p_1,p_2,x) =s(g_1)=t(g_2)$ the point where the arrows join and recall the maps $Q(\cdot,p)$ and $\tilde Q(\cdot,p)$ yielding inverses for $s(\cdot,p)$ and $t(\cdot,p)$, respectively, for $p\sim 0$ as introduced above.
%
We can then arrive to the main formula for $\sigma^c$.
\begin{lemma}\label{lem:cansigcoord}
With the notations above,
\begin{eqnarray}
J^*\sigma^c|_{(p_1,p_2,x)} &=&  \gS(p_1,p_2,x) \tilde \mu\otimes \tilde \mu \otimes \mu. \nonumber\\
\text{with}&&\nonumber\\
\gS(p_1,p_2,x)&=&|det(\partial_x \tilde x|_{(p_1,p_2,x)})\cdot det(\partial_x Q|_{(\tx,p_1)}) \cdot det(\partial_x \tilde Q|_{(\tx,p_2)}) |^{1/2}  \label{eq:Jcansigma}
\end{eqnarray}
\end{lemma}

\begin{proof}
Let us simplify the notation $G=G_\pi$ within this proof and follow Definition \ref{def:canonsigma} of $\sigma^c$.
As a first step, consider the following map which gives an alternative parametrization of composable arrows,
\begin{equation*}\label{eq:defPhi}
\Phi:U_\Phi \subset M^* \times M \times M^* \to G_\pi\compo , \ \Phi(p_1,\tilde x,p_2) = (g_1=(Q(\tilde x,p_1),p_1), g_2=(\tilde Q(\tilde x,p_2),p_2) ).
\end{equation*}
Note that $\Phi(p_1,\tilde x,p_2)=J(p_1,p_2,x)$ where $\tilde x= s(J(p_1,p_2,x))$ defines a bijection $x \mapsto \tx$ for small $p_1,p_2\sim 0$, as noted before. Consider basis $[M],[M^*]$ on which $\mu$ and $\tilde \mu$ take the value $1$, and define
\[ \beta' := D_{(p_1,\tx,p_2)}\Phi ( [M^*] \times [M] \times [M^*] ) \text{ basis for $T_{(g_1,g_2)}G\compo$.}\]
We complement this basis with a linearly independent set $\beta''\subset T_{g_1}G \oplus T_{g_2}G$ given by 
\[ \beta'' = \{ (\dot g_1, \dot g_2): \dot g_1 =(\dot x_1,\dot p_1)=(\partial_x Q|_{(\tx,p_1)}\dot x,0), \dot g_2=0 ; \dot x \in [M]\}. \]
We can then apply Definition \ref{def:canonsigma} and compute
\[ \sigma^c(\beta') = \frac{(\lambda_G\otimes \lambda_G)(\beta'\cup \beta'')}{\mu((Ds_1-Dt_2)(\beta''))} \]
which straightforwardly yields
\begin{equation*}\label{eq:canonicalenhonGpi}
(\Phi^*\sigma^c)|_{(p_1,\tx,p_2)} = |det(\partial_x Q|_{(\tx,p_1)})\cdot det(\partial_x \tilde Q|_{(\tx,p_2)}) |^{1/2} \tilde \mu\otimes \mu \otimes \tilde \mu \in \Gamma |T(M^*\times M\times M^*)|^{1/2}.
\end{equation*}
Finally, to get to the $J$-parametrization of composable arrows, we compute
\[ J^*\sigma^c = (\Phi \Phi^{-1}J)^*\sigma^c = (\Phi^{-1}J)^* (\Phi^*\sigma^c) \]
from which the Lemma follows by observing $\Phi^{-1}J(p_1,p_2,x) = (p_1,\tx,p_2)$ with $\tx$ seen as a function of $(p_1,p_2,x)$.
\end{proof}

We now discuss formulas for an enhancement $\sigma=\sigma^{a_0}$ of $G_\pi$ defined by a factor $a_0$ coming from a star product in the form \eqref{eq:FIOstar}. We recall that $\sigma^{a_0}$ is defined by such an integral operator following the general procedure in \cite[\S 8.5]{GSbook}. Specializing to our particular setting, one obtains
\begin{equation}\label{eq:Jsigma}
(J^*\sigma^{a_0})|_{(p_1,p_2,x)} = a_0(p_1,p_2,x) \ \tmu \otimes \tmu \otimes \mu.
\end{equation}
(The details can be found in \cite{Ledthesis}.)
Combining with Lemma \ref{lem:cansigcoord} and the fact that
$$\gS(0,0,x) = 1, \ \forall x\in M,$$ 
we thus get the following relation
\begin{equation}\label{eq:sigmaa0simple}
\sigma^{a_0}|_{(g_1,g_2)} = \frac{a_0(p_1,p_2,x_3)}{\gS(p_1,p_2,x_3)} \sigma^c|_{(g_1,g_2)},   
\end{equation}
where the factor $\gS$ was defined in \eqref{eq:Jcansigma} and $(g_1,g_2)=J(p_1,p_2,x_3) \in G_\pi\compo$ with $J$ defined in \eqref{eq:defJ}. We also note that we can recover the $a_0$ factor from the half-density $\sigma^{a_0}$ via
\[ a_0(p_1,p_2,x) = \sigma^{a_0}|_{J(p_1,p_2,x)}\left( D_{(p_1,p_2,x)}J([M^*]\times [M^*] \times [M]) \right) \]
for $[M],[M^*]$ dual basis on which $\mu$ and $\tilde\mu$ take the value $1$.

\begin{remark} (Associativity of $\sigma^{a_0}$)
One can verify directly that $\sigma^{a_0}$, as defined above, satisfies the associativity condition for half-densities \eqref{eq:assocsigma} iff $a_0$ satisfies the equation \eqref{eq:Ea0} obtained alternatively from a stationary phase argument. Moreover, underlying an associative $\sigma^{a_0}$ we have
\begin{equation}\label{eq:deffa0S}
 f_{a_0}(p_1,p_2,x) = \frac{a_0(p_1,p_2,x_3)}{\gamma_S(p_1,p_2,x_3)}
\end{equation}
which, seen as a function on $G_\pi\compo$ through $J^{-1}$, must satisfy eq. \eqref{eq:assocf}. (See \cite{Ledthesis} for more details.) When $a_0\neq 0$ near $0\times 0 \times M$, following Corollary \ref{cor:H2}, we call $f_{a_0}$ the \textbf{multiplicative 2-cocycle defined by $a_0$}, $\delta_{\C^*}f_{a_0}=1$.
\end{remark}

\begin{remark} (Convolution of enhanced horizontal bisections) Let us go back to the convolution operation of Remark \ref{rmk:variosconvolbisec}. For any function $F\in C^\infty(M)$ let us denote the corresponding horizontal Lagrangian $L_F=\{(x,\partial_xF|_x):x\in M\}$ in $T^*M$. The projection $q:T^*M\to M$ induces a diffeomorphism $q_F:=q|_{L_F}:L_F\simeq M$ for any $F$. Thinking of $p_1,p_2\in M^*\subset C^{\infty}(M)$ as linear functions, when they are close enough to zero, we get
\[ (\gr(m),\sigma^{a_0})\circ \left((L_{p_1},q^*_{p_1}(f_1\mu))\times (L_{p_2},q^*_{p_2}(f_2\mu))\right) = \left(L_{S(p_1,p_2,\cdot)}, q^*_{S(p_1,p_2,\cdot)}[f_1(\partial_{p_1}S) f_2(\partial_{p_2}S) a_0|_{(p_1,p_2,\cdot)} \mu]\right)  \]
where $\mu=|dx|^{1/2}$ as above and $f_j\in C^\infty(M), \ j=1,2$. Note that this is the $\h\to 0$ stationary phase approximation of $(f_1 e^{\fih p_1}) \star_\h (f_2e^{\fih p_2})|_x$, with $\star_\h$ as in \eqref{eq:FIOstar}. With respect to this operation, the choice of the canonical $\sigma^c$ has the following special property. Denote $s_p = s|_{L_p}:L_p \overset{\sim}{\to} M$ for $p\sim 0$, and similarly $t_p=t|_{L_p}$. Recall the parameterization $J(p_1,p_2,x)=(g_1,g_2)$ and denote $\tilde{x}_{p_1,p_2}: M\to M, \ x \mapsto \tilde x(p_1,p_2,x)=s(g_1)=t(g_2)$ the corresponding meeting point of the arrows. Then, we get the identities
\begin{eqnarray}
(\gr(m),\sigma^c)\circ \left((L_{p_1},s_{p_1}^*(f_1\mu))\times (L_{p_2},t^*_{p_2}(f_2\mu))\right) = \left(L_{S(p_1,p_2,\cdot)}, \tilde x_{p_1,p_2}^*[f_1 f_2 \mu]\right),\nonumber
\\
(\gr(m),\sigma^c)\circ \left((L_{p_1},s_{p_1}^*(f_1\mu))\times (L_{p_2},s^*_{p_2}(f_2\mu))\right) = \left(L_{S(p_1,p_2,\cdot)}, \tilde x_{p_1,p_2}^*(f_1) \cdot s|_{L_{S(p_1,p_2,\cdot)}}^*(f_2 \mu)\right).
\label{eq:convolsigmac}
\end{eqnarray}
(And a similar one for pullbacks along $t_{p_j}$.)
Observe that, when $f_1$ is invariant along the leaves of $(M,\pi)$, the second case above yields the half-density $s|_{L_{S(p_1,p_2,\cdot)}}^*(f_1f_2 \mu)$ on the rhs.
\end{remark}

\subsection{Characterizing Kontsevich's half-density}\label{subsec:Khalf}
Here, we recall Kontsevich's star product $\star^K$ with structure given in \eqref{eq:kontstarfactors} and focus on the underlying semiclassical factors $(S_K,a_0^K)$. We first describe the formal family of symplectic groupoids $G_K$ underlying $S_K$ following \cite{C22}. Second, the main result (Theorem \ref{thm:Kenhancement}) states that the {\bf Kontsevich enhancement} $\sigma^K\equiv \sigma^{a^K_0}$ of $G_K$ defined by the factor $a_0^K$ is equivalent, as a formal family, to the canonical enhancement $\sigma^c$. 
Within this subsection, then, we shall also work with \emph{formal families} of symplectic groupoids integrating the formal family $\epsilon \pi \in \mathfrak{X}^2(M)[[\epsilon]]$ of Poisson structures, with $\epsilon$ a formal parameter which we distinguish from $\h$ for conceptual reasons. A reference for the precise definitions and their relation to asymptotic expansions is \cite[\S 4]{C22}. 


\titpar{The formal family of enhanced symplectic groupoids underlying $\star^K$} 

In this subsection, folowing \cite[\S 4.1 and 4.2]{C22}, we consider 
\begin{equation}\label{eq:defGK} 
(G_K \rra M, \omega_c)\text{: the \emph{formal family of symplectic groupoids} obtained from $G_{\e\pi}$ around $\e=0$}
\end{equation}
which underlies Kontsevich's $\star^K_h$ and the novel object of study:
\begin{equation}\label{eq:defKenhanc}
\text{the \textbf{Kontsevich enhancement} }\sigma^K\equiv \sigma^{a_0^K}\text{ defined by } a_0^K=e^{K_{1-loop}}\text{ through formula \eqref{eq:Jsigma}.}
\end{equation}
We recall that the groupoid structure in $G_K$ is defined by the zero-loop expansion $S_K$ in \eqref{eq:kontstarfactors} while the enhancement corresponds to the $1$-loop expansion $K_{\text{1-loop}}$ (see also \cite{CatDheFel}).

Let us first describe the formal family \eqref{eq:defGK} above. The idea is that every structure map, seen as the operation of pullback of functions, is a formal expansion in $\e$:
\[ s^*,t^*: C^\infty(M) \to C^\infty(T^*M)[[\e]], m^*: C^\infty(T^*M)\to C^\infty(G_K\compo)[[\e]].\]
These are obtained by asymptotic expansion of the structure maps of the construction $G_{\e\pi}$ as $\e\to 0$.
We thus see $G_K$ as a family of symplectic groupoids, parameterized by a formal parameter $\e$, which integrates $(M,\e\pi\in \mathfrak{X}(M)[[\e]])$.
More details about formal families of groupoids can be found in \cite[\S 4]{C22} and \cite{CatDheFel} in terms of expansions, in \cite{Karabegov} for their relation to general star-products, and in \cite{CDh} for the source (realization) map. 
At the level of generating functions, we have
\cite[Thm. 4.13]{C22} which says that the asymptotic Taylor expansion of $S_{\epsilon\pi}$ at $\epsilon=0$ reproduces Kontsevich's $0$-loop factor appearing in \cite{CatDheFel},
\begin{equation}
Taylor_{\epsilon=0} S_{\epsilon \pi} =S_K.
\end{equation}
These facts will be enough for our study below. (See also more details in \cite[\S 4]{C22} about how the asymptotic expansion in $\e$ is equivalent to asymptotic expansions around $p_1,p_2=0$ which appear in \cite{Karabegov}.)

Let us now discuss the Konstsevich enhancement $\sigma^K$ of \eqref{eq:defKenhanc}. It can also be seen as a formal $\e$-family of enhancements of $G_{\e\pi}$ defined through the formula $\sigma^{a_0^K}$ of eq. \eqref{eq:Jsigma}, where $$a_0^K=e^{K_\text{1-loop}}\in C^\infty(M^*\times M^*\times M)[[\e]]$$ 
is a formal expansion obtained as the sum over Kontsevich $1$-loop graphs $K_\text{1-loop}$  with coefficients being the symbols of the corresponding bidifferential operators for the Poisson structure $\e\pi$ (see \cite{Kont,CatDheFel}). The following properties of $\sigma^K$ follow from its definition. Since $\star^K$ is associative, then $a_0^K \in C^\infty(M^*\times M^* \times M)[[\e]]$ defines a formal family of solutions of eq. \eqref{eq:Ea0}. It follows that $\sigma^K$ defines a formal family of associative enhancements of $(G_K\rra M,\omega_c)$, in the sense of Definition \ref{def:main}.
The relation between $\sigma^K$ and the canonical enhancement $\sigma^c$ defined by the coordinate $\mu=|dx|^{1/2}$ is $\sigma^K = f_{a_0^K}\cdot \sigma^c$,
where the formal family of multiplicative $2$-cochains $f_{a_0^K}$ is defined as in \eqref{eq:deffa0S}, 
\[f_{a_0^K}=\frac{a_0^K}{\gS} \in C^\infty(M^*\times M^*\times M)[[\e]],\]
with $\gS$ defined in \eqref{eq:Jcansigma} with underlying $\e$-families of structure maps. 
In particular, $f_{a_0^K}$ is a forma $\e$-family of multiplicative 2-cocycles,
\[ \delta_{\C^*} f_{a_0^K} = 1.\]
(Recall that we use $J$ defined in \eqref{eq:defJ} to identify small composable arrows with elements of $M^*\times M^*\times M$.)
Moreover, by the definition \eqref{eq:Jcansigma} of $\gS$, we have $\gS(0,0,x)=1$ to all orders in $\e$ so that that $\sigma^K$ is of exponential type relative to $\mu=|dx|^{1/2}$ (Definition \ref{def:exptype}),
\[ f_{a_0^K} = e^{h^K}, \ h^K=K_{\text{1-loop}}-ln(\gS) \in C^\infty(M^*\times M^*\times M)[[\e]],\]
where $ln$ is the natural logarithm defined near $1\in\C$. It follows that $\delta h^K=0$ yields  additive 2-cocycles (recall $\delta$ from \eqref{eq:addidelta}) in the formal family $G_K$ and we thus call $h^K$ the {\bf additive Kontsevich semiclassical 2-cocycle}.

\titpar{Characterizing trivial cocycles}
Since we aim at comparing $\sigma^K$ and $\sigma^c$ in terms of exp-equivalence (recall Definition \ref{def:exptype}), we study here a useful characterization of trivial $\e$-families of additive $2$-cocycles.

Let us recall the setting. In the formal family $G_K$ given by \eqref{eq:defGK}, all the structure maps are formal $\e$-expansions. We then have $C^\infty(G_K)=C^\infty(T^*M)[[\e]]$ as $1$-cochains and the following isomorphism for additive $2$-cochains
\[ J^*:C^\infty(G_K\compo)\overset{\sim}{\to} C^\infty(M^*\times M^*\times M)[[\e]] \]
which we omit from the notation and which is defined by $\e$-expansion of the corresponding map \eqref{eq:defJ} for the family $S_{\e\pi}$. The additive cohomology differential $\delta$ of \eqref{eq:addidelta} similarly involves formal $\e$-expansion of the differential for $G_{\e\pi}$.
\begin{proposition}\label{prop:trivcocy}
Let $h\in C^\infty(M^*\times M^*\times M)[[\e]]$ a formal family of additive $2$-cocycles,
\[ \text{ $\delta h =0$  to all orders in $\e$,}\] 
which also satisfies the hypothesis
\begin{equation}\label{eq:symmcondit}
 \frac{\partial^2 h}{\partial {p_1}\partial {p_2}}|_{(0,0,x)} \text{ is a \emph{symmetric} bilinear form on $M^*$ to all orders in $\e$ and $\forall x\in M$.}
\end{equation}
Assume further that $h$ is normalized, namely, $h(0,p,x)=h(p,0,x)=0$ to all orders in $\e$. Then, $h = \delta h'$ for some formal family of 1-cochains $h'\in C^\infty(M^*\times M)[[\e]]$.
\end{proposition}

\begin{proof}
The key idea is to use a perturbation method together with the van Est isomorphism at $\e=0$, as follows. Consider the initial element $(M,\pi=0)$ in the $\e$-family which is integrated by $(G_0=T^*M\rra M,\omega_c)$ as described in Example \ref{ex:pi=0}. The underlying Lie algebroid is $A_0 = T^*_{\pi=0} M \to M$ with trivial anchor and bracket. The van Est isomorphism (\cite{Cra03}) in this particular case yields 
$$v_0:H^{2}(G_0) \overset{\sim}{\to} H^2(A_0)=\mathfrak{X}^2(M).$$ 
On the other hand, in this case, at the level of cochains $h_0\in C^\infty(T^*M\times_M T^*M)$ we have the following description of the van Est map,
\[ v_0(h_0) =  skew( \frac{\partial^2 h_0}{\partial {p_1}\partial {p_2}}|_{(0,0,x)} )\]
where $\partial_{p_j}$ denote the vertical derivatives at the zero section and $skew$ means the skew-symmetrization of the bilinear form on the fibers of $T^*M$. We thus have the following fact: for $h_0 \in C^\infty(T^*M\times_M T^*M)$,
\[(a): \text{ if $\delta_0h_0=0$ and $\frac{\partial^2 h_0}{\partial {p_1}\partial {p_2}}|_{(0,0,x)}$ is symmetric $\forall x$ then $\exists h_0'\in C^\infty(T^*M)$ such that $h=\delta_0 h_0'$,} \]
where $\delta_0$ is the additive differential on $G_0$, i.e. formula \eqref{eq:addidelta} with composition being $+$ on the fibers of $T^*M$. Moreover, when $h_0$ is normalized, then $h'_0$ can be chosen so that it is at least quadratic in the $p$-variables:
\[ \text{$h'_0(x,0) = 0$ (normalized) and  $\partial_p h'_0|_{(x,0)} = 0$.} \]
This fact holds by considering the Taylor expansion of such $h'(x,p)$ around $p=0$: the $p$-constant term is ruled out by the normalization hypothesis $f(0,0,x)=0$ while a possible linear term in $f'$ can be removed since such terms are in the kernel of $\delta_0$.

Let us now come back to a general family $h=\sum_{n\geq 0} h_n \e^n$ as in the hypothesis of this Proposition. We write $\delta = \sum_{n\geq 0} \e^n D_n$ (with $D_0=\delta_0$ as above) taking into account the $\e$-expansion of the structure maps of the underlying $G_{\e\pi}$. The condition \eqref{eq:symmcondit} reads $v_0(h_n)=0$ for each $n$ while the cocycle condition $\delta h=0$ reads
\[ D_0 h_n = - \sum_i D_{n-i} h_i.\]
On the other hand, a necessary condition for $h=\delta h'$ with $h'=\sum_n h'_n \e^n$ is
\[ D_0 h'_n = h_n - \sum_{i\leq n-1} D_{n-i}h'_i =:\tilde h_n \]
which provides a recursive formula for the $h'_n$. Using the fact (a) above, it is enough to recursively check that the $\tilde h_n$ can be chosen so that
\[ (i): \ D_0\tilde h_n =0, \ (ii): \ v_0(\tilde h_n)=0.\]
The recursion starts with $h_0$ which satisfies (i,ii) as a direct consequence of the hypothesis on $h$. Assume (i,ii) hold for $h_i, i<n$. The first condition $D_0 \tilde h_n=0$ follows by direct computation from the identities coming from expanding $\delta^2=0$ together with $\delta h=0$ (see also \cite{Ledthesis}); this verifies (i). For (ii), we use (a) above, with the mentioned refinement for normalized cochains, and notice that the proof shall be finished once we show the following general fact:
\[\text{(b): if $f'(x,p)$ is at least quadratic in $p$, then $v_0(D_kf')=0, \ \forall k\geq 0$.}\]
To show this, we can go back to the non-formal family with $\e\geq 0$ with generating function $S_{\e}\equiv S_{\e\pi}$ and compute
\[\delta f'|_{(p_1,p_2,x)} = f'(\partial_{p_1}S_\e(p_1,p_2,x),p_1) + f'(\partial_{p_1}S_\e(p_1,p_2,x),p_1) - f'(x, \partial_x S_\e(p_1,p_2,x)). \]
We need to show that $b^{ij}(x):= \partial_{t_1=0}\partial_{t_2=0} [D_\e f'|_{(t_1 e^i,t_2 e^j,x)}]$
is symmetric in $i,j$, where $(e^i)$ is any linear basis of $M^*$. After $\partial_{t_1=0}\partial_{t_2=0}$, the first two terms in  $\delta f'|_{(p_1,p_2,x)}$ vanish due to the quadratic behaviour of $f'$ on the $p$ variable. Similarly, the last term under $\partial_{t_1=0}\partial_{t_2=0}$ will only have contributions from the linear terms in $p_1$ and $p_2$ of $\partial_x S_\e(p_1,p_2,x)$. In turn, these terms are $p_1 + p_2$ by \cite[eq. 23]{C22}, so that $b^{ij}(x)$ is indeed symmetric. This finishes the proof of $(b)$ and, hence, of the Proposition.
\end{proof}

In the following, we apply this characterization to analyze $\sigma^K$.

\titpar{The equivalence between $\sigma^K$ and $\sigma^c$}

We want to study how far is Kontsevich's enhancement $\sigma^K$ from the canonical one $\sigma^c$. Since the canonical enhancement satisfies the associativity properties structurally, without any other input, such a comparison can "explain" the associativity properties behind the 1-loop factor of Kontsevich's formula $\star^K_h$. 

To apply Proposition \ref{prop:trivcocy} above to the additive Kontsevich semiclassical cocycle $h=h^K$, we need to show that it satisfies condition \eqref{eq:symmcondit}.

\begin{lemma}\label{lem:symm}
Both $\e$-families of additive 2-cochains $K_{1-loop}$ and $ln(\gS)$ are normalized and satisfy the symmetry condition \eqref{eq:symmcondit}.
\end{lemma}

\begin{proof}
Following the description of $K_{1-loop}$ in terms of Kontsevich diagrams, see \cite{Kont} and \cite{CatDheFel}, we have 
\[ K_{1-loop}(p_1,p_2,x) = \e^2 b_x(p_1,p_2) + E_\e(p_1,p_2)  \]
where $b_x$ is a bilinear form on $M^*$ (defined by the simplest K-graph with 1-loop which has 2 aereal vertices) and where $E_\e(p_1,p_2)$ are terms at least cubic in the $p_j$'s and in $\e$. From this it follows that $K_{1-loop}$ is normalized and satisfies \eqref{eq:symmcondit}.

Next, let us verify that $ln(\gS)$ is normalized. To this end, let us first consider $\gS(0,p_2,x)$. Using $Q(\tx,0)=\tx$ and $\tilde Q(\tx(0,p_2,x),p_2)=x$ we obtain
\[ \gS(0,p_2,x) = |\underset{1}{\underbrace{det(\partial_x Q|_{(\tx,0)})}} \cdot \underset{1}{\underbrace{det(\partial_x \tilde x|_{(0,p_2,x)})\cdot  det(\partial_x \tilde Q|_{(\tx,p_2)})}} |^{1/2} = 1. \]
Similarly, one obtains $\gS(p_1,0,x)=1$, and then $ln(\gS)$ is normalized.
Finally, we need to check that $ln(\gS)$ satisfies \eqref{eq:symmcondit}. Using the normalization identities just proven, we have
\[ b^S_x:=\frac{\partial^2 (ln(\gS))}{\partial p_1 \partial p_2}|_{(0,0,x)} = \frac{1}{\gS(0,0,x)}  \frac{\partial^2 \gS }{\partial p_1 \partial p_2}|_{(0,0,x)} = \frac{\partial^2 \gS }{\partial p_1 \partial p_2}|_{(0,0,x)}. \]
On the other hand, the desired symmetry of the bilinear form $b^S_x$ follows directly from the following property with respect to inversion $inv(x,p)=(x,-p)$, 
\[ (a): \gS (p_1,p_2,x) = \gS(-p_2,-p_1,x). \]
It thus remains to show $(a)$, for which we recall the general identity
\[ m(g_1,g_2)=g_3 \iff m(inv(g_2),inv(g_1)) = inv(g_3), \]
implying
\[Q(\tx,-p_2) = \tilde Q(\tx,p_2), \ \tilde Q(\tx,-p_1)=Q(\tx, p_1), \ \tx(p_1,p_2,x)=\tx(-p_2,-p_1,x). \]
Identity $(a)$ follows directly by the above and the definition \eqref{eq:Jcansigma} of $\gS$, thus finishing the proof.
\end{proof}

We are now ready to state the main theorem of this Section.
\begin{theorem}\label{thm:Kenhancement}
Let $(M\simeq \R^n,\pi)$ be a coordinate Poisson manifold and $(G_K\rra M,\omega_c)$ the formal family of symplectic groupoids given in \eqref{eq:defGK}. Consider $\sigma^K$ the Kontsevich enhancement defined by the 1-loop diagrams factor $a^K_0 = e^{K_{1-loop}}$ in \eqref{eq:defKenhanc}. 
Then, $\sigma^K$ is equivalent, in the sense of Definition \ref{def:equivenhanc}, to the canonical enhancement $\sigma^c$ associated with the coordinate half-density $\mu=|dx|^{1/2}$. More specifically, there exists  $h'\in C^\infty(M^*\times M)[[\e]]$ such that
\[ \sigma^K = e^{\delta h'} \sigma^c. \]
\end{theorem}
The proof follows directly by using Lemma \ref{lem:symm} and Proposition \ref{prop:trivcocy} applied to the formal 2-cocycle $h^K=K_{1-loop}-ln(\gS)$, since $\sigma^K=e^{K_{1-loop}-ln(\gS)} \sigma^c$.
Note that, in the notation of Definition \ref{def:equivenhanc}, the factor implementing the equivalence is $\factor = e^{h'}$ so that $\sigma^K$ and $\sigma^c$ are also exp-equivalent in the sense of Definition \ref{def:exptype} along the formal family $G_K$. The deformation class of Remark \ref{rmk:defo} is trivial since $h=\delta h'$ is exact in this case and this is coherent with the fact that the underlying Kontsevich-class family of Poisson structures is simply $\h\pi$ (see \cite{Kont}).

\begin{remark} (Further study of $h'$)
First, we remark that $h^K=\sum_{n\geq 2} h_n \e^n$ has the special property that  $h_n(x,p_1,p_2)$ is a homogeneous polynomial of degree $n$ in $(p_1,p_2)$. This follows from the description of the underlying $1$-loop Kontsevich diagrams and the homogeneity properties of $S_{\e\pi}$ in \cite[Cor. 3.33]{C22}. In such a case, following the proof of Proposition \ref{prop:trivcocy}, one can verify that $h'=\sum_n h'_n \e^n$ can be chosen so that $h'_n(x,p)$ is also homogeneous of degree $n$ in $p$. About the leading $\e^2$-terms, we can compare that of $K_{1-loop}$ (which corresponds to the simplest Kontsevich diagram with 1-loop) to that of $ln(\gS)$. These can be seen as symmetric bilinear forms on $M^*$ where $\pi$ enters quadratically and the case $M=\g^*$ below suggests the conjecture that they are  always equal. In such case, the possible corrections in $K_{1-loop}$ to $ln(\gS)$ must be of higher order than $\e^2$. This will be explored elsewhere.
\end{remark}

\begin{remark} (Formal path integral computation)
Let us recall from \cite{CatFel} that $f_1\star^K_\h f_2|_x$ can be obtained as a path integral $I_\h$ in the Poisson Sigma Model (PSM) and that the expansion \eqref{eq:kontstarfactors} corresponds to the asymptotic expansion of $I_\h$. This integral behaves like an (infinite dimensional) oscillatory integral with oscillatory phase given by the PSM action $\mathcal{A}$ (including source terms, \cite[\S 5.1]{C22}). Following a formal stationary phase approximation, $\mathcal{A}$ evaluated on its critical points was already shown in \cite{C22} to yield the canonical generating function $S_\pi$ above. It then makes sense to continue this formal computation and try to obtain a functional formula for the factor $a_0^K$. It should correspond to the square-root of the Hessian factor in the underlying stationary phase formula and we can compare it to our canonical factor $\gS$. Note that, for $M=\g^*$ below, these two coincide. This functional study will be carried out elsewhere.
\end{remark}

\subsection{The case of a linear Poisson structure}\label{subsec:linear}
Let $\g$ be a Lie algebra and $(M=\g^*,\pi)$ the associated linear Poisson manifold as in Example \ref{ex:pilinear}.
In this final subsection, we illustrate the structural understanding of the corresponding Kontsevich enhancement $\sigma^K$ in this case. Here, the semiclassical data $(S_K,\sigma^K)$ determines $\star^K$ completely and this particular star product plays an important role in understanding the Duflo isomorphism (see \cite[\S 8.3]{Kont} and \cite{Amar}) and its extension to convolution algebras of invariant distributions proposed by Kashiwara-Vergne (\cite{KV,AnSaTo}). The main result (Proposition \ref{prop:gastsigmaK}) says that $\sigma^K=\sigma^c$, thus providing an interpretation for the key square-root Jacobian factors appearing in $a_0^K$ within our general theory of associative half-densities. 

\titpar{A class $\star^F$ of star products and the corresponding data $(S,\sigma^F)$}

Following \cite{Amar}, we consider a family of star products $\star^F_\h$ on $M=\g^*$ parameterized by a function $F:\g \to \C$ with $F(0)=1$, as follows:
\begin{equation}\label{eq:starFgast}
f_1\star^F_\h f_2 (x) = (2\pi \h)^{-n}\int_{p_1,p_2 \in \g} \F(f_1)(p_1) \F(f_2)(p_2) \ \frac{F(p_1)F(p_2)}{F(p_1\cdot p_2)} \ e^{\frac{i}{\h} x(p_1\cdot p_2)} \ dp_1 dp_2.
\end{equation}
In the formula, $f_1,f_2 \in C^\infty(M)$ are functions, $n=dim(\g)$, $x\in M=\g^*$ and $M^*=\g$, 
$$\F(f)(p) = (2\pi \h)^{-n/2} \int_{x\in M} f(x) e^{-\frac{i}{\h} x p } \ dx$$ 
is the $\h$-scaled Fourier transform and $dx,\ dp$ denote the Lebesgue measures. Moreover, $\g \ni p_1,p_2 \mapsto p_1\cdot p_2 = p_1 + p_2 + \frac{1}{2} [p_1,p_2] + \dots$ is the BCH-series for the Lie algebra $\g$ which induces a local group structure on a neighborhood $\G \subset \g$ of zero.
We can restrict to considering $f_j$ being Schwartz functions so that their Fourier transform $\F(f_j), j=1,2$ has rapid decay at infinity. If we only care about the asymptotic expansion as $\h\to 0$, following \cite{Amar} we can further require $\F(f_j)$ to have compact support near $p=0$.

For any such $F$, $\star^F$ satisfies the axioms (S1,2,3) of the introduction, as follows. The change of variables $p_j'=\h p_j, \ j=1,2$ makes (S1,2) more evident and leaves the precise expressions as appearing in \cite{Amar}. The associativity (S3) follows from the associativity of the BCH product on $\G\subset \g$ appearing in 
\[ S(p_1,p_2,x) := x(p_1\cdot p_2) \]
and from the factor 
\[a_0^F(p_1,p_2,x) := \frac{F(p_1)F(p_2)}{F(p_1\cdot p_2)} \]
being $x$-independent and trivially a multiplicative $2$-cocycle on the local group $\G\ni p_1,p_2$. At the semiclassical level, we have the underlying local symplectic groupoid $(G_\pi = T^*\G\rra \g^*, \omega_c)$ given by the cotangent lift of $\G$, as recalled in Example \ref{ex:pilinear} (see also \cite[Ex. 2.7]{C22}). The underlying {\bf semiclassical data} $(S,a_0^F)$ consists of the function $S$ satisfying the SGA-equation \eqref{eq:SGA} and $a_0^F$ satisfying the equation \eqref{eq:Ea0} (since $\partial^2_x S=0$ in this case). We also remark that $S=S_\pi$ coincides with the canonical generating function for $G_\pi=T^*\G$ (see \cite[Ex. 3.12]{C22}), and that $(S,a_0^F)$ completely determines $\star^F$. 

Following the general theory, 
$a_0^F$ determines an {\bf underlying associative enhancement} $\sigma^F$ of $G_\pi$ via eq. \eqref{eq:Jsigma}. Finally, using the fact that the projection $T^*\G\to \G$ is a morphism of local groupoids, we can write the identity of $2$-cochains
\[ a_0^F = \delta_{\C^*}F, \]
where $F$ is seen as a $1$-cochain on $\G\subset \g$ and the pullback is omitted.

\titpar{Kontsevich's enhancement equals the canonical one}
Next, we recall from \cite{Amar} the fact that the family $\star^F$ contains 3 important star products quantizing $M=\g^*$ as particular cases. Consider the functions
\begin{equation}\label{eq:Fsvarios}
 F_G(p) = 1, \ F_R(p) = det\left(\frac{sinh(\frac{1}{2}ad_p)}{\frac{1}{2}ad_p} \right), \ F_K(p) = det\left(\frac{sinh(\frac{1}{2}ad_p)}{\frac{1}{2}ad_p} \right)^{1/2} = F_R(p)^{1/2},
 \end{equation}
where  $ad_p:\g\to \g$ is the adjoint action of $p\in \g$ and $sinh(z)=(e^z-e^{-z})/2$ is the hyperbolic sine function. Then, the corresponding star products
\[ \star_h^{F_G}, \ \star_h^{F_R}, \ \star_h^{F_K} \]
reproduce the so-called \textbf{Gutt star product}, \textbf{Rieffel star product} and \textbf{Kontsevich star product}, respectively. The fact that the asymptotic expansion of the above integral formula reproduces Kontsevich's general formula $\star^K_h$ when specialized to $M=\g^*$,
\[ \star_h^{F_K}=\star^K_\h\]
 is non-trivial (see \cite{Sho} and more details in \cite{Amar}).

Now, all the star products in the class $\star^F$ are equivalent to each other (see the definition of equivalence in \cite{Kont}). In other terms, all the enhancements $\sigma^F$ of $T^*\G$ are equivalent in the sense of Definition \ref{def:equivenhanc} via $\kappa = F'/F$. Nevertheless, among this class, Kontsevich's $F=F_K$ has a non-trivial distinctive property: $f_1\star^K_\h f_2=f_1f_2$ when $f_1,f_2$ are $ad^*$-invariant polynomials on $M=\g^*$ (see \cite{Amar} and \cite{AnSaTo} for its extension to suitable convolution algebras).

The following main result of this subsection provides an interpretation, within our theory of associative half-densities, for the formula of the underlying factor $a_0^K$, namely, that $\sigma^K$ is no other than the canonical half-density on $T^*\G$.

\begin{proposition}\label{prop:gastsigmaK}
Let $F=F_K$ be the function corresponding to Kontsevich star product $\star^{F_K}_h=\star^K_h$, as given by formula \eqref{eq:Fsvarios}. Consider $\sigma^K\equiv \sigma^{a_0^K}$ the corresponding enhancement of $(G_\pi\rra M,\omega_c)$ defined by $a_0^K=a_0^{F_K}$ via \eqref{eq:Jsigma}. Then,
\[ \sigma^{K} = \sigma^c\text{ the canonical enhancement associated with $\mu=|dx|^{1/2}$.} \]
\end{proposition}

\begin{proof}
Unwinding the definitions, we need to show that $ a_0^{F_K} = \gS$, where the factor $\gS$ was defined in \eqref{eq:Jcansigma}. 
%
For the local Lie group $\G\subset \g$ defined by the BCH structure, we denote $\theta_L,\theta_R \in \Omega^1(\G,\g)$ the left and right invariant Maurer-Cartan forms. 
Following e.g. \cite[Sec. 1.5]{DK}, we have the following classical formulas
\[\theta_L|_p:T_p \G = \g \to \g, \ v\mapsto \left(\frac{1-e^{-ad_p}}{ad_p} \right)(v) = e^{-\frac{1}{2}ad_p}\left(\frac{sinh(\frac{1}{2}ad_p)}{\frac{ad_p}{2}} \right) (v) ,\]
and $\theta_R|_p = Ad_p(\theta_L|_p)$.
Note that 
$$\tilde F(p) := det(\theta_L|_p) =  det \left(\frac{1-e^{-ad_p}}{ad_p} \right) = det(e^{-\frac{1}{2}ad_p}) F_R(p)$$ 
is the Jacobian factor appearing in the left invariant Haar measure on $\G$ (or in another integration $\tilde \G$ via exponential coordinates) and with $F_R(p)$ defined in \eqref{eq:Fsvarios}.

We are now ready to compute $\gS$ explicitly using formula \eqref{eq:Jcansigma}. We will consider $p,p_1,p_2\sim 0$ small enough.
Following the definition of the maps $Q,\tilde Q$ and $\tx$, 
%
we can then evaluate $\gS$ yielding
%
\[ \gS(p_1,p_2,x) = | det(D_{p_1}R_{p_2}\circ D_0L_{p_1}) \ \tilde F(p_1) \ det(Ad_{p_2}) \tilde F(p_2)|^{1/2}.\]
For the first factor, it is easy to check form the identity on any Lie group $g_1e^{t\xi}g_2=g_1g_2g_2^{-1}e^{t\xi}g_2$ that
\[ D_{p_1}R_{p_2}\circ D_0L_{p_1} = D_0L_{p_1\cdot p_2} \circ Ad_{-p_2}. \]
Then, 
\[ det(D_{p_1}R_{p_2}\circ D_0L_{p_1}) = det(D_0L_{p_1\cdot p_2})det(Ad_{-p_2}) = \frac{det(Ad_{-p_2})}{det(\theta_L|_{p_1\cdot p_2})}, \]
so that, after cancellation of $Ad_{p_2}$ and $Ad_{-p_2}$ factors using $p^{-1}=-p$ in the BCH structure,
\[\gS(p_1,p_2,x) = \left|\frac{\tilde F(p_1) \tilde F(p_2)}{\tilde F(p_1\cdot p_2)}\right|^{1/2}. \]
Finally, using \cite[Lemma 3.1]{Amar} (see also \cite[proof of Prop. 3.2]{Amar}) we know that the $det(e^{-\frac{1}{2}ad_p})$ factor between $\tilde F(p)$ and $F_R(p)$=$F_K(p)^2$, as defined in \eqref{eq:Fsvarios}, cancels in the above combination, namely,
\[ \frac{\tilde F(p_1) \tilde F(p_2)}{\tilde F(p_1\cdot p_2)} =\frac{F_R(p_1) F_R(p_2)}{F_R(p_1\cdot p_2)} \]
so that
\[ \gS(p_1,p_2,x) = \left|\frac{F_R(p_1) F_R(p_2)}{F_R(p_1\cdot p_2)}\right|^{1/2}=\frac{F_K(p_1) F_K(p_2)}{F_K(p_1\cdot p_2)}=a_0^{F_K}(p_1,p_2,x),\]
for $p_1,p_2\sim 0$ small enough, as wanted. This finishes the proof.
\end{proof}

In particular, the above result provides a purely semiclassical interpretation for the factor $F_K$ appearing in the Duflo isomorphism as a correction to the PBW map (see \cite{Kont,Amar}). Namely, the factor $F_K$ can be interpreted as the equivalence factor $\factor=F_K/1$, in the sense of Definition \ref{def:equivenhanc}, between the enhancement  $\sigma^{F_G}$ underlying the Gutt star product (which is induced by PBW) and the canonical enhancement $\sigma^c=\sigma^K$.

%
%
%
%

\appendix 
\section{Half-densities with values in a line bundle}\label{app:E}

Following the constructions in semiclassical analysis \cite{GSbook,Mein}, half-densities for Lagrangians $\Lambda \subset T^*X$ in cotangent bundles are in general not $\C$-valued but take values in the underlying Maslov line bundle $\LL(\Lambda)$ over $\Lambda$. When considering more general symplectic ambients in place of $T^*X$, such as $\Lambda=\gr(m)\subset \overline{G}\times \overline{G} \times G$ for a symplectic groupoid $G$, one can consider half-densities taking values in more general line bundles $E\to \Lambda$ which are considered part of the defining data. In this Appendix, we explore this setting and isolate the key characteristic properties from the case $E=\LL(\gr(m))$ appearing when $G\simeq T^*M$ is a local symplectic groupoid. We also show some general properties of the resulting $E$-valued enhancements and observe that, for considerations that only involve the germ around the units of $G$, one can restrict to the $\C$-valued case.

%
%
%
%

\titpar{The case of the Maslov line for local symplectic groupoids}\label{maslovmult}

We first consider $(G\rra M,\omega)$ a {\bf local symplectic groupoid} in a given germ class around its units, following the conventions of \cite{CabMarSal1} as in \S \ref{subsec:coordstar}. In such a germ class, we can always take a representative with ambient $G=T^*M$, with the units to be given by the zero section $1_x=0_x, \ x\in M$ and with $\omega=\omega_c$ the canonical symplectic form. We denote $G^{(2)} \subset G\times G$ the open subset of composable arrows which are also small enough to be multiplied in the local groupoid (this domain is part of the defining data). Similarly $G^{(3)}$ is an open subset inside composable triples which are small enough so that associativity $g_1(g_2g_3)=(g_1g_2)g_3$ holds. The nerve $G^{(k)}$ can be constructed similarly recalling the requirement $1^{(k)}_M \subset G^{(k)}, \ k\geq 2$ where $1^{(k)}_x=(1_x,\dots,1_x)\in G^k$.

We now recall the following facts about the {\bf Maslov line bundle} $\LL(\Gamma)\to \Gamma$ associated with a canonical relation $\Gamma: T^*X_1 \dto T^*X_2$, following \cite{GSbook} (see also \cite{Mein}). This is a specialization of a Maslov line bundle $\LL(\Lambda)\to \Lambda$ associated with any Lagrangian submanifold on a cotangent bundle $\Lambda \subset T^*X$ which is defined by the relative position of the two Lagrangians, $T_z \Lambda$ and the cotangent fiber, at each point $z\in T^*X$ inside the symplectic vector space $T_z(T^*X)$. We shall not need to recall the full definition but only an important functoriality property, as follows (see \cite[\S 5.13.5]{GSbook}). Given two canonical relations $\Gamma_j:T^*X_j \dto T^*X_{j+1}, \ j=1,2$ which are cleanly composable, then
\begin{equation}\label{eq:functLLGamma}
\alpha^*\LL(\Gamma_2\circ \Gamma_1) \simeq pr_1^*\LL(\Gamma_1) \otimes pr_2^*\LL(\Gamma_2)
\end{equation}
where $\alpha:\Gamma_2\star \Gamma_1:=\{(z_1,z_2,z_3): (z_1,z_2)\in \Gamma_1, \ (z_2,z_3)\in \Gamma_2\} \to \Gamma_2\circ \Gamma_1$ is given by $\alpha(z_1,z_2,z_3)=(z_1,z_3)$, as in Section \ref{sec:half}, and the projections $pr_j:\Gamma_2\star \Gamma_1 \to \Gamma_j$ are given by $pr_j(z_1,z_2,z_3)=(z_j,z_{j+1})$ for $j=1,2$.

\medskip

Combining the above two facts, for a local symplectic groupoid with $G=T^*M$ and $\omega=\omega_c$, as above,  there is a natural line bundle
\begin{equation}\label{eq:defEm} 
E_m := i_{(2)}^* \LL(\gr(m)) \text{ over } G^{(2)}
\end{equation}
where $\gr(m): G\times G \dto G$ is seen as a canonical relation $T^*M \times T^*M = T^*(M\times M) \dto T^*M$ and $i_\comp{2}:G^\comp{2}\overset{\sim}{\to} \gr(m), \ (g_2,g_1)\mapsto (g_2,g_1,g_2g_1)$. 
%
%
Moving towards an associativity property for $E$, let us introduce the following open embeddings:
\[ i_{3(21)}: G^\comp{3} \to \gr(m)\circ (\gr(id_G)\times \gr(m)), (g_3,g_2,g_1) \mapsto (g_3,g_2,g_1,g_3(g_2g_1)), \]
\[ i_{(32)1}: G^\comp{3} \to \gr(m)\circ (\gr(m) \times \gr(id_G)), (g_3,g_2,g_1) \mapsto (g_3,g_2,g_1,(g_3g_2)g_1). \]
We also recall the simplicial face maps $d_i:G^{(n+1)}\to G^{(n)},\ i=0,..,n+1$ which remove the $i$-th vertex $x_i$ in a sequence of (small) composable arrows
\begin{equation}\label{eq:Gstring}
 x_{n+1} \overset{g_{n+1}}{\leftarrow} x_{n} \overset{g_{n}}{\leftarrow} \dots  x_1 \overset{g_1}{\leftarrow} x_0 
 \end{equation}
and either composes the corresponding arrows or erases the ones on the extremes. 

\begin{lemma}
With the notations above,
\[ i^*_{3(21)} \LL(\gr(m)\circ (\gr(id_G)\times \gr(m))) \simeq d_3^*E_m \otimes d_1^*E_m, \]
\[ i^*_{(32)1} \LL(\gr(m)\circ (\gr(m)\times \gr(id_G))) \simeq d_2^*E_m \otimes d_0^*E_m, \]
as line bundles over $G^\comp{3}$.
\end{lemma}
%
This lemma follows directly from the functoriality property mentioned above and the details are left to the reader.

To provide a simplicial interpretation of this property, given a line bundle $E$ over $G^{(n)}$, we denote
\[  \dE (E) = d_0^*E \otimes d_1^*E^* \otimes d_2^*E \otimes d_3^*E^* \otimes ...   \text{, the resulting line bundle over $G^{(n+1)}$.}\]
Notice that this is the multiplicative version of the additive differential $\delta = \sum_{i=0}^{n+1} (-1)^i d_i^*$ where the minus is replaced by taking the dual bundle $E^*$.
In this way, we can state the following associativity property for $E$.
\begin{proposition}
For a local symplectic groupoid $(G\rra M,\omega)$ with $G=T^*M$, $\omega=\omega_c$ and $1_x=0_x$, the Maslov line $E_m$ over $\gr(m)$ given in \eqref{eq:defEm} satisfies 
$$\delta_l (E_m) \simeq \underline{\C}\text{ the trivial line over }G^\comp{3}.$$
Moreover, $E_m|_{1\compo_M}\simeq \C$.
\end{proposition}
The first part follows directly from the above Lemma. The second follows from the fact that $1_x=0_x$ and $m\circ 1\compo_x = 1_x$, so that the relative position of $T_{1\compo_x}\gr(m)$ and the cotangent fiber inside $T_{1\compo_x}(T^*M^3)$ does not change with $x\in M$, allowing for a trivialization of the restricted Maslov bundle.

\titpar{Associative line bundles over composable arrows}
Let us now consider a general (global) Lie groupoid $G\rra M$. 
The above result motivates considering more general line bundles which satisfy the following definitions.

A line bundle $E\to G^{(2)}$ is called {\bf associative} if it satisfies the multiplicative 2-cocycle condition \begin{equation}\label{eq:deltaE=1}
\dE(E) \simeq \uC\text{ the trivial line bundle over $G^{(3)}$.}
\end{equation}
Note that this means that $d_0^*E \otimes d_2^*E \simeq d_1^*E \otimes d_3^*E$, as before.
Similar notions appeared for the germ around units of the pair groupoid $G=M\times M$, for example, in \cite{Melrose} as \emph{local line bundles} and related to star products on symplectic $M$. The $\dE(E)$ operation also appears in connection to bundle gerbes \cite{Murray} for $G$ being a submersion groupoid.
From the simplicial point of view, $\dE$ defines a multiplicative line bundle-valued version of differential cochains. In this context, we say that a line bundle $E$ over $G^{(n+1)}$ is {\bf normalized} if
\[ s_j^* E \simeq \C, \forall j \]
where the degeneracy maps $s_j:G^{(n)}\to G^{(n+1)}, \ j=0,..,n$ insert an identity connecting $x_j$ and $x_{j+1}$ in the string \eqref{eq:Gstring}. We highlight the following property.
\begin{lemma}
Let $E$ over $G\compo$ be an associative line bundle. Then, $E$ is normalized iff $E|_{1\compo_M}\simeq \uC$ is trivial along the units.
\end{lemma}
The proof can be done straightforwardly with similar arguments to those used for Lemma \ref{lem:sigmanonvanish}.
We remark that $E$ over $G\compo$ being associative and normalized does not imply that it is a trivial bundle $E\nsimeq \uC$ in general.

\titpar{The associativity condition for $E$-valued half-densities}
Let $(G\rightrightarrows M,\omega)$ be a symplectic groupoid and we come back to the study of enhancements of $\gr(m)$. Consider a line bundle $E$ over $G\compo\simeq \gr(m)$. An {\bf $E$-valued enhancement} of $(G\rra M,\omega)$ is an $E$-valued half-density along the graph of the multiplication map
\begin{equation}\label{eq:sigmaE}
\sigma \in \Gamma\left(|T\gr(m)|^{1/2}\otimes E\right). 
\end{equation}
Using $\gr(m)\simeq G\compo$, we shall identify $\sigma$ with a section of $|TG\compo|^{1/2}\otimes E$. When $E$ is associative, the associativity equation \eqref{eq:assocsigma} for such a $\sigma$ is well defined since $d_0^*E \otimes d_2^*E \simeq d_1^*E \otimes d_3^*E$ and, in this case, we say that $\sigma$ is an {\bf associative $E$-valued enhancement}.

To analyze existence, consider the sequence \eqref{eq:shortseq} and denote $v_1:G\compo \to M$ the assignment of the object $x=s(g_1)=t(g_2)$ where a composable $(g_1,g_2)$ meets. Similarly to the construction of canonical half-densities in Section \ref{sec:half}, a nowhere-vanishing section 
\[ \mu \in \Gamma\left( |v_1^*TM|^{1/2}\otimes E^* \right), \ \mu\neq 0 \]
can be checked to give rise to an associative enhancement through the formula $\sigma^c = (\lambda_G\otimes \lambda_G)/\mu$, where now $\C$-multiplication is also used for the $\C$-line values.
Nevertheless, such a $\sigma^c$ will be nowhere-vanishing and this is impossible when $E\nsimeq \uC$  is non-trivial (since the bundle of half-densities itself is always trivializable). On the other hand, when $E|_{1\compo_M}\simeq \uC$ (equiv. when $E$ is normalized, by the preceding results), if we are only interested in the germ of $G$ around the units, we can assume $E\simeq \uC$ is trivial.

From this discussion, we see that $E$-valued enhancements can encode non-trivial information which is global in $G$ (far away from units) and that it can lead to non-trivial twists in the geometry. In the context of quantization, such global aspects can be relevant for considerations that are not microlocal (as $\h\to 0$), for example, for associativity of long words in the sense of \cite{CFer2} or when considering quantization of compact symplectic manifolds (e.g. with complex polarizations). These aspects will be explored elsewhere.

\end{document}